\documentclass[11pt, a4paper, english]{amsart}
\usepackage{amsmath} 
\usepackage{amsthm} 
\usepackage{amssymb} 
\usepackage{amscd} 
\usepackage[dvipsnames]{xcolor}

\usepackage{array} 

\usepackage[cal=boondoxo,scr=euler]{mathalfa}
\usepackage[mathscr]{eucal} 
\usepackage[backref=page]{hyperref} 
\usepackage{caption} %
\usepackage{graphics,graphicx} 
\usepackage{xcolor}
\usepackage{tikz,tikz-cd}
\usetikzlibrary{arrows.meta, backgrounds, hobby, shapes.geometric,}

\usepackage{enumerate} 
\usepackage{newtxtext,newtxmath} 
\usepackage[margin=1.30in]{geometry}
\usepackage{verbatim}
\usepackage{xcolor}
\makeatletter
\@namedef{subjclassname@2020}{\textup{2020} Mathematics Subject Classification}
\makeatother

\newcommand{\stirling}[2]{\biggl[\genfrac{}{}{0pt}{}{#1}{#2}\biggr]}

\hypersetup{
    colorlinks = true,
    linkbordercolor = {white},
    linkcolor = {BrickRed},
    anchorcolor = {black},
    citecolor = {BrickRed},
    filecolor = {cyan},
    menucolor = {red},
    runcolor = {cyan},
    urlcolor = {magenta}
}

\usetikzlibrary{automata}

\newtheoremstyle{teoremas}
{12pt}
{13pt}
{\itshape}
{}
{\bfseries}
{}
{.5em}
{}

\theoremstyle{teoremas}
\newtheorem{teo}{Theorem}[section]
\newtheorem{prop}[teo]{Proposition}
\newtheorem{lem}[teo]{Lemma}
\newtheorem{cor}[teo]{Corollary}
\newtheorem{conj}[teo]{Conjecture}

\theoremstyle{definition}
\newtheorem{defi}[teo]{Definition}
\newtheorem{ej}[teo]{Example}
\newtheorem{rem}[teo]{Remark}

\DeclareMathOperator{\ehr}{ehr}

\DeclareMathOperator{\R}{\mathbb{R}}

\newcommand{\Z}{\mathbb{Z}}
\newcommand{\vol}{\operatorname{vol}}

\title{Lattice points in slices of prisms}

\author[L.~Ferroni]{Luis~Ferroni}

\author[D.~McGinnis]{Daniel~McGinnis}
\thanks{Luis~Ferroni was partially supported by the Swedish research council, grant 2018-03968. Daniel~McGinnis was partially supported by the Iowa State University Department of Mathematics through the Lambert Graduate Assistantship}

\address{Department of Mathematics, KTH Royal Institute of Technology, Stockholm, Sweden} 

\email{ferroni@kth.se}

\address{Department of Mathematics, Iowa State University, Ames, Iowa} 

\email{dam1@iastate.edu}

\allowdisplaybreaks

\begin{document}

\begin{abstract}
    We conduct a systematic study of the Ehrhart theory of certain slices of rectangular prisms. Our polytopes are generalizations of the hypersimplex and are contained in the larger class of polypositroids introduced by Lam and Postnikov; moreover, they coincide with polymatroids satisfying the strong exchange property up to an affinity. We give a combinatorial formula for all the Ehrhart coefficients in terms of the number of weighted permutations satisfying certain compatibility properties. This result proves that all these polytopes are Ehrhart positive. Additionally, via an extension of a result by Early and Kim, we give a combinatorial interpretation for all the coefficients of the $h^*$-polynomial. All of our results provide a combinatorial understanding of the Hilbert functions and the $h$-vectors of all algebras of Veronese type, a problem that had remained elusive up to this point. A variety of applications are discussed, including expressions for the volumes of these slices of prisms as weighted combinations of Eulerian numbers; some extensions of Laplace's result on the combinatorial interpretation of the volume of the hypersimplex; a multivariate generalization of the flag Eulerian numbers and refinements; and a short proof of the Ehrhart positivity of the independence polytope of all uniform matroids.
\end{abstract}


\maketitle

\section{Introduction}

\subsection{Overview} In the present article we will focus on the integer point enumeration of certain polytopes that arise as a result of slicing a \emph{rectangular prism}. To be precise, if $\mathbf{c}=(c_1,\ldots,c_n)$ is a vector of positive integers, we define the $n$-dimensional rectangular prism
    \[ \mathscr{R}_{\mathbf{c}} := [0, c_1] \times [0, c_2] \times \cdots \times [0, c_n].\]
This is of course a lattice polytope in $\mathbb{R}^n$. We will be dealing with certain \emph{slices} of rectangular prisms. More precisely, we define for each $k\in \mathbb{Z}_{>0}$ and $\mathbf{c}$ as before:
    \begin{equation}\label{eq:defi-prism}
    \mathscr{R}_{k,\mathbf{c}} := \left\{ x\in \mathscr{R}_{\mathbf{c}} : \sum_{i=1}^n x_i = k\right\}.
    \end{equation}

The problem of enumerating the lattice points lying inside a convex polytope $\mathscr{P}\subseteq \mathbb{R}^n$ with integral vertices is of fundamental importance in several areas within mathematics and has been systematically approached under different perspectives. Whenever a lattice polytope $\mathscr{P}$ is fixed, the function that associates to each positive integer $t$ the number of integral points lying in the dilation $t\mathscr{P}$, namely, 
    \[ t \mapsto \#(t\mathscr{P}\cap \mathbb{Z}^n),\]
happens to be a polynomial of degree $d = \dim\mathscr{P}$. This polynomial, first studied by Ehrhart \cite{ehrhart}, is known as the \emph{Ehrhart polynomial} of the polytope $\mathscr{P}$.

Ehrhart polynomials have proven to be a remarkably useful tool in different problems within discrete geometry and algebraic combinatorics, especially due to their connection with chromatic polynomials of graphs, order polynomials of posets, solutions of certain linear diophantine equations, Hilbert functions in commutative algebra and toric varieties in algebraic geometry.

For some specific classes of polytopes such as regular simplices, hypercubes, cross-polytopes and zonotopes, it is possible to determine explicitly the Ehrhart polynomial and, more specifically, its coefficients. For example, the Ehrhart polynomial of a rectangular prism admits a quite easy expression:
    \begin{equation}\label{eq:ehr_prism}
    \ehr(\mathscr{R}_{\mathbf{c}}, t) = \prod_{i=1}^n (c_it+1).\end{equation}
A classical result in the theory of Ehrhart polynomials is that for every lattice polytope $\mathscr{P}$ of dimension $d$, if we write
    \begin{equation}\label{eq:ehrh-general}\ehr(\mathscr{P}, t) = a_d t^d + a_{d-1} t^{d-1} + \cdots + a_{1}t + a_0,\end{equation}
then some of the coefficients are easy to understand, i.e., $a_0 = 1$, $a_{d} = \vol(\mathscr{P})$ and $a_{d-1} = \frac{1}{2} \vol(\partial\mathscr{P})$; see \cite{beck-robins} for detailed proofs of this fact and further exploration on the connections mentioned above. In particular, from this we see that the Ehrhart polynomial of a polytope as an invariant is a vast generalization of the volume.

It is our belief that the study in this paper will be of interest and relevance in all of the following frameworks:
    \begin{itemize}
        \item Ehrhart positivity.
        \item Refinements and generalizations of the (flag) Eulerian numbers.
        \item Polymatroids satisfying the strong exchange property and polypositroids.
        \item Combinatorial interpretations of $h^*$-polynomials.
        \item Hilbert functions of arbitrary algebras of Veronese type.
    \end{itemize}

In what follows we expand on the relation of our work with each of the above points.

\subsection*{Ehrhart positivity} 

By looking at equation~\eqref{eq:ehrh-general} we see that the coefficients of the terms of degree $d$, $d-1$ and $0$ are positive for every lattice polytope. However, the remaining coefficients can be negative in general. Although there are some general formulas for computing them \cite{mcmullen}, they are quite intricate. One of the major questions in the study of Ehrhart theory is to determine classes of polytopes having the property that \emph{all} of the coefficients of $\ehr(\mathscr{P}, t)$ are positive; such polytopes are said to be \emph{Ehrhart positive}. See \cite{liu} for a thorough exposition regarding positivity of Ehrhart polynomials.

\begin{teo}\label{thm:ehrhart-positivity}
    All slices of prisms $\mathscr{R}_{k,\mathbf{c}}$ are Ehrhart positive.
\end{teo}

Sometimes we informally refer to the polytopes $\mathscr{R}_{k,\mathbf{c}}$ as ``thin'' slices of prisms. This is because if we fix $\mathbf{c}=(c_1,\ldots,c_n)$ and two nonnegative integers $a < b$, by defining
    \begin{equation}\label{eq:fat-slice} \mathscr{R}'_{a,b,\mathbf{c}} := \left\{ x\in \mathscr{R}_{\mathbf{c}} : a \leq \sum_{i=1}^n x_i \leq b\right\},\end{equation} 
we obtain another type of slice of a prism. We usually refer to them as ``fat'' slices of a prism. The difference between these two types of slices will not be of much relevance; we will see explicitly how it is possible to transform one into the other while preserving the Ehrhart polynomials. In particular, we will also prove that fat slices of prisms are Ehrhart positive.

In fact, the Ehrhart positivity will be deduced via a combinatorial formula for each of the coefficients that reveals positivity. To formulate this statement, we introduce the notion of ``$\mathbf{c}$-compatible weighted permutations'' (defined in Section~\ref{CompatibilitySection}), where $\mathbf{c}=(c_1,\ldots,c_n)$ is a vector of positive integers. We define
    \[ W(\ell,n,m,\mathbf{c}) := \# \left\{\text{$\mathbf{c}$-compatible weighted $\sigma\in\mathfrak{S}_n$ with $m$ cycles and weight $\ell$}\right\}.\]

\begin{teo}\label{main}
    Let $\mathbf{c}=(c_1,\dots,c_n)\in \mathbb{Z}_{>0}^n$ and $0 < k < c_1+\cdots+c_n$. For each $0\leq m\leq n-1$, the coefficient of $t^m$ in $\ehr(\mathscr{R}_{k,\mathbf{c}},t)$ is given by
    \[
    [t^m]\ehr(\mathscr{R}_{k,\mathbf{c}},t) = \frac{1}{(n-1)!}\sum_{\ell = 0}^{k-1}W(\ell,n,m+1,\mathbf{c})A(m,k-\ell-1).
    \]
    In particular Theorem~\ref{thm:ehrhart-positivity} holds.
\end{teo}

Here $A(n,k)$ stands for the Eulerian numbers, i.e., the number of permutations $\sigma\in\mathfrak{S}_n$ having exactly $k$ descents.

Part of the intuition on how to properly define  the notion of weights and compatibility comes from \cite{ferroni1} and \cite{hanelyetal}. In the first of these papers, a rather easy particular case of the above description yields the Ehrhart positivity of $\Delta_{k,n} = \mathscr{R}_{k,(1,\ldots,1)}$, and the role of $W(\ell,n,m+1,\mathbf{c})$ in Theorem~\ref{main} is played by the ``weighted Lah numbers'' $W(\ell,n,m+1)$. That explains much of the terminology and the notation that we use in the present paper. On the other hand, in \cite{hanelyetal} a different proof for the Ehrhart positivity of hypersimplices is outlined, the main advantage being that it does not rely on generating function identities but only on an inclusion-exclusion argument which is applied to a family of cleverly defined sets. Recently, the ideas used in the Ehrhart positivity results of these papers have also been extended and adapted in \cite{mcginnis2023combinatorial} to establish the Ehrhart positivity of different classes of polytopes.

\subsection*{Flag Eulerian numbers}

Notice that \emph{a~priori} it is not clear what the volume of a slice of a prism will look like. A famous result attributed in \cite{stanley-eulerian} to Laplace states that if we slice the $n$-dimensional unit cube, which is just $\mathscr{R}_{\mathbf{c}}$ for $\mathbf{c} = (1,\ldots,1)\in \mathbb{Z}^n$, with the hyperplane $\sum_{i=1}^{n} x_i= k$, the normalized volume of the resulting polytope is precisely the Eulerian number $A(n-1,k-1)$, that is, the number of permutations $\sigma\in \mathfrak{S}_{n-1}$ that have exactly $k-1$ descents. In the literature, the set of points in the $n$-dimensional unit cube having sum of coordinates equal to $k$ is known as the \emph{hypersimplex} $\Delta_{k,n}$. 

We introduce the following generalization of the notion of flag Eulerian numbers in \cite{han-josuat}.

\begin{defi}
    The flag Eulerian number $A_{n,k}^{(\mathbf{c})}$ is defined as the number of $\mathbf{c}$-colored permutations having $k$ flag descents.
\end{defi}

(The notions of $\mathbf{c}$-colored permutations and flag descents will be explained later, in Section~\ref{sec:6}.) In analogy to the case of the classical Eulerian numbers, we will realize these numbers as volumes of polytopes.

\begin{teo}
    Let $\mathbf{c}=(c_1,\dots, c_{n})\in \mathbb{Z}^n_{>0}$, and let $\mathbf{c}'=(c_1,\dots,c_{n},1)$. The normalized volume of $\mathscr{R}_{k+1,\mathbf{c}'}$ equals the flag Eulerian number $A_{n,k}^{(\mathbf{c})}$. 
\end{teo}

In particular, this extends some results by Han and Josuat-Verg\`es \cite{han-josuat} which were valid for vectors $\mathbf{c}$ of the form $(r,\ldots,r)$. Since the volume is also captured as the leading coefficient of the Ehrhart polynomial, as a consequence of Theorem~\ref{main}, we can conclude a handy formula for the volume of $\mathscr{R}_{k,\mathbf{c}}$, and thus for the flag Eulerian numbers.

\begin{cor}\label{coromain}
    Let $\mathbf{c}=(c_1,\dots,c_n)\in \mathbb{Z}_{>0}^n$. Then the volume of $\mathscr{R}_{k,\mathbf{c}}$ is given by
        \[ \vol(\mathscr{R}_{k,\mathbf{c}}) = \frac{1}{(n-1)!}\sum_{\ell = 0}^{k-1} B(\ell,\mathbf{c})A(n-1,k-\ell-1)\]
    where $B(\ell,\mathbf{c})$ is defined as the number of ways of placing $\ell$ indistinguishable balls into $n$ boxes of capacities $c_1-1, \ldots, c_n-1$ respectively.
\end{cor}

In particular, when combining the preceding two results we obtain a combinatorial identity relating the flag Eulerian numbers with the classical Eulerian numbers.

\begin{cor}
    The flag Eulerian number $A_{n,k}^{(\mathbf{c})}$ is given by
    \[A_{n,k}^{(\mathbf{c})} = \sum_{\ell = 0}^{k} B(\ell, \mathbf{c}) A(n, k-\ell)\]
    where $B(\ell,\mathbf{c})$ is defined as in Corollary \ref{coromain}.
\end{cor}

\subsection*{Polymatroids and polypositroids}

The study of the volume of the hypersimplex has been a \emph{leit motiv} for several important developments in combinatorics. For instance, the hypersimplex $\Delta_{k,n}$ admits a regular unimodular triangulation as was proved by Stanley in \cite{stanley-eulerian} (see also \cite{sturmfels}). Generalizing such triangulations, in \cite{lam-postnikov} Lam and Postnikov introduced the notion of alcoved polytope, a family that essentially contains all hypersimplices and, more importantly, all the polytopes of the form $\mathscr{R}_{k,\mathbf{c}}$. They proved that all alcoved polytopes possess a regular unimodular triangulation and therefore their volume can be recovered by counting the number of simplices of the triangulation, see \cite[Theorem 3.2]{lam-postnikov}.

The edges of a polytope of the form $\mathscr{R}_{k,\mathbf{c}}$ are parallel to some vector of the form $e_i-e_j$, hence they are particular instances of generalized permutohedra or polymatroids. Actually, in the influential paper \cite{herzog-hibi} these polytopes are coined ``polymatroids of Veronese type''. The main result of Herzog, Hibi and Vladoiu in \cite{herzog-hibi-vladoiu} characterizes them (up to an affinity) as the polymatroids satisfying the strong exchange property.

More recently, Lam and Postnikov studied in \cite{lam-postnikov-polypositroids} the class of \emph{polypositroids}, the family of all polytopes that are simultaneously polymatroids and alcoved polytopes. The polytopes $\mathscr{R}_{k,\mathbf{c}}$ are polypositroids, hence an additional combinatorial toolbox is at disposal for their study. In the past two decades, the Ehrhart theory of the hypersimplex and other alcoved polytopes and polymatroids has been matter of intensive research, motivated in part due to the conjectures posed by De~Loera, Haws and K\"oppe in \cite{deloera-haws-koppe}. In \cite{postnikov} Postnikov proved the Ehrhart positivity of certain polymatroids arising as Minkowski sums of simplices. In \cite{ferroni1} Ferroni gave a combinatorial formula for the Ehrhart coefficients of the hypersimplex $\Delta_{k,n}$ (see also \cite{hanelyetal}). In \cite{castillo-liu,castillo-liu2} Castillo and Liu proved the positivity of the Ehrhart coefficients of high degree and the linear term for arbitrary polymatroids; the latter was also proved in \cite{jochemko-ravichandran}. Although arbitrary polymatroids can have negative Ehrhart coefficients \cite{ferroni3}, it is conjectured in \cite{fjs} that positroids are indeed Ehrhart positive. We remark that it is not true that general alcoved polytopes are Ehrhart positive, as in fact the class of order polytopes fails to be Ehrhart positive as was shown by Stanley (see \cite{liu-tsuchiya}).

\subsection*{Combinatorial interpretations for \texorpdfstring{$h^*$}{h*}-polynomial-vectors}
The study of the $h^*$-polynomial (which encodes the same information as the usual Ehrhart polynomial) of alcoved polytopes is an intriguing and very active area of research, see for instance recent work \cite{fjs,sinn-sjoberg}. The $h^*$-polynomial of $\Delta_{k,n}$ exhibits remarkable combinatorial properties, and has been described by Li \cite{li} and by Early \cite{early} and Kim \cite{kim} using different approaches. This problem had also been studied by Katzman in \cite{katzman}, where the problem of giving a combinatorial interpretation for this remained widely open. 

Due to Stanley's result \cite{stanley-hstar} that shows that the coefficients of the $h^*$-polynomial of a polytope always are nonnegative integers, it is highly desirable to provide combinatorial interpretations for these numbers. The second main result of this paper consists of such a combinatorial interpretation for the polytopes $\mathscr{R}_{k,\mathbf{c}}$. We will generalize Kim's result to all slices of prisms.

\begin{teo}\label{main-hstar}
    Let $\mathbf{c}=(c_1,\ldots,c_n)\in\mathbb{Z}_{>0}^n$ and $0<k<c_1+\cdots+c_n$. For each $0\leq i\leq n$, the coefficient $[x^i]h^*(\mathscr{R}_{k,\mathbf{c}},x)$ equals the number of $\mathbf{c}$-compatible decorated ordered set partitions of type $(k,n)$ having winding number $i$.
\end{teo}

Again, we defer to Section~\ref{sec:5} the definitions and the terminology for the above statement. On the other hand, since the sum of the entries of the $h^*$-vector of a polytope equals the normalized volume, then each of the above coefficients provides a refinement for the flag Eulerian numbers.

\subsection*{Hilbert functions and algebras of Veronese type}

Our results admit an interesting interpretation from the perspective of commutative algebra. Let us fix a field $\mathbb{F}$, a vector $\mathbf{c}=(c_1,\ldots,c_n)\in\mathbb{Z}^n_{>0}$ and a positive integer $k<c_1+\cdots+c_n$. Consider $\mathscr{V}(\mathbf{c},k)$, the graded subalgebra of $\mathbb{F}[x_1,\ldots,x_n]$ generated by all the monomials $x_1^{\alpha_1}\cdots x_n^{\alpha_n}$ where $\alpha_1+\cdots+\alpha_n=k$ and $\alpha_i\leq c_i$ for each $1\leq i\leq n$. We say that the algebra $\mathscr{V}(\mathbf{c},k)$ is of \emph{Veronese type}. There is considerable literature on this topic, including \cite{bruns}, \cite{denegri-hibi}, \cite{katzman} and \cite{herzog-zhu}. 

Since $\mathscr{A}:=\mathscr{V}(\mathbf{c},k)$ is isomorphic to the Ehrhart ring of $\mathscr{R}_{k,\mathbf{c}}$ we can relate the Ehrhart polynomial with the \emph{Hilbert function}, i.e.,
    \[ \ehr(\mathscr{R}_{k,\mathbf{c}}, m) = \dim_{\mathbb{F}}(\mathscr{A}^m),\]
where $\mathscr{A}^m$ is the graded component of degree $m$ of $\mathscr{A}$.

\begin{cor}
    Let $\mathscr{A}$ be an arbitrary algebra of Veronese type. Then the Hilbert function of $\mathscr{A}$ is a polynomial with positive coefficients.
\end{cor}

In addition to the positivity of the coefficients, of course we get the combinatorial interpretation provided by Theorem~\ref{main}. On the other hand, we also obtain an interpretation for the coefficients of their $h$-vectors, i.e., the numerators of the Hilbert series.

\begin{cor}\label{coro:veronese}
    The $i$-th entry of the $h$-vector of the algebra of Veronese type $\mathscr{V}(\mathbf{c},k)$ over the field $\mathbb{F}$ counts the number of $\mathbf{c}$-compatible decorated ordered partitions of type $(k,n)$ and winding number $i$.
\end{cor}

\subsection{Outline}

The paper is structured as follows. In Section~\ref{sec:2} we describe the basics of the Ehrhart theory of both ``thin'' and ``fat'' slices of prisms and provide a first explicit formula for the Ehrhart polynomial. Although this formula, stated in Theorem~\ref{thm:first-formula} does not reveal the positivity of the coefficients, it happens to be interesting on its own. In Section~\ref{sec:3} we introduce the notion of ``$\mathbf{c}$-compatible weighted permutations'' following \cite{ferroni1} and \cite{hanelyetal}, which allows us to define the numbers $W(\ell,n,m+1,\mathbf{c})$ appearing in Theorem~\ref{main}. The definition is somewhat involved, and the enumeration of such objects, achieved in Theorem~\ref{thm:technical} is arguably the most technical part of the paper. In Section~\ref{sec:4} we state and prove Theorem~\ref{main}. In Section~\ref{sec:5} we approach the $h^*$-polynomials of $\mathscr{R}_{k,\mathbf{c}}$ inspired by \cite{early} and \cite{kim}; the main difficulty with this section is coming up with the right definitions, but once this is done, Kim's techniques can be applied in a rather direct manner.  Finally, in Section~\ref{sec:6} we discuss some applications, in particular we prove Corollary \ref{coromain} and discuss Corollary \ref{coro:veronese}; we extend results by Han and Josuat-Verg\`es \cite{han-josuat} by combining our results and a generalized version of their notion of flag Eulerian number; finally, also in Section~\ref{sec:6}, we give a short proof of the Ehrhart positivity of the independence polytope of all uniform matroids.

\section{Slices of prisms}\label{sec:2}

The main goal in this section is to provide a handy expression for the Ehrhart polynomial of the slices of prisms. This is done in Theorem~\ref{thm:first-formula}. The main drawback is that we cannot see the coefficients directly and we cannot even be sure that they are positive. Nevertheless, this concrete expression will allow us to unveil a nice factorization in Section~\ref{sec:4} that provides us with the combinatorial formula of Theorem~\ref{main}.

\subsection{Thin and fat slices}

By definition, the polytope $\mathscr{R}_{k,\mathbf{c}}$ introduced in equation \eqref{eq:defi-prism} is the intersection between an $n$-dimensional rectangular prism and a hyperplane. It is an $(n-1)$-dimensional polytope, unless $k = 0$ or $k = c_1+\cdots+c_n$: in these two cases it collapses to a point; also, it will be certainly empty if $k > c_1+\cdots+c_n$.

\begin{ej}
    Consider the vector $\mathbf{c}=(6,3,4)$ and take $k = 7$. The polytope $\mathscr{R}_{k,\mathbf{c}}$ is obtained as the intersection of a rectangular prism in $\mathbb{R}^3$ and the hyperplane $x+y+z=7$, as depicted on the left in Figure \ref{fig:1}. We can see that it corresponds to a pentagon in $\mathbb{R}^3$ with vertices  $(6,0,1)$, $(6,1,0)$, $(4,3,0)$, $(0,3,4)$ and $(3,0,4)$, shown on the right of Figure \ref{fig:1}. Furthermore, notice that the combinatorics of this polytope is genuinely different to that of all hypersimplices having dimension two.
\begin{figure}\centering
    \begin{tikzpicture}[scale=0.63]
    \draw[black][blue,top color=blue, bottom color=blue, fill opacity=.25] (3,0,0) -- (3,0,6) -- (3,4,6) -- (3,4,0) -- cycle;
    \draw[black][blue,top color=blue, bottom color=blue, fill opacity=.25] (0,4,0) -- (3,4,0) -- (3,4,6) -- (0,4,6) -- cycle;
    \draw[black][blue,top color=blue, bottom color=blue, fill opacity=.25] (0,0,6) -- (0,4,6) -- (3,4,6) -- (3,0,6) -- cycle;
    
    \filldraw [black] (0,0,0) circle (.75pt);
    \filldraw [black] (4,0,0) circle (0pt) node[below] {$y$};
    \filldraw [black] (0,5,0) circle (0pt) node[left] {$z$};
    \filldraw [black] (0,0,7) circle (0pt) node[below] {$x$};
    \draw[densely dotted, ->] (0,0,0) -- (4,0,0);
    \draw (3,0,6) -- (0,0,6);
    \draw[densely dotted,->] (0,0,0) -- (0,0,7);
    \draw (3,4,0) -- (3,4,6);
    \draw (3,4,6) -- (0,4,6);
    \draw (0,4,6) -- (0,4,0);
    \draw[densely dotted,->] (0,0,0) -- (0,5,0);
    \draw (3,0,0) -- (3,4,0);
    \draw (3,0,6) -- (3,4,6);
    \draw[black] (3,0,6) -- (3,0,0);
    \draw[black] (0,4,6) -- (0,0,6);
    \draw[black] (3,4,0) -- (0,4,0);
    
    \fill[green,top color=green, bottom color=green, fill opacity=0.1] (0,4,3) -- (3,4,0) -- (0,7,0) -- cycle; 
    
    \filldraw [black] (4,4,-1) circle (0pt) node[right] {$x+y+z=7$};
    
    \draw[blue] (0,4,3) -- (3,4,0);
    \draw[blue] (0,4,3) -- (0,1,6);
    \draw[blue] (0,1,6) -- (1,0,6);
    \draw[blue] (1,0,6) -- (3,0,4);
    \draw[blue] (3,0,4) -- (3,4,0);
    \fill[green,top color=green, bottom color=green, fill opacity=0.1] 
    (0,0,7)--(4,-4,7)--(4,3,0)--(3,4,0)--(3,0,4)--(1,0,6)--(0,1,6);
    \fill[green,top color=green, bottom color=green, fill opacity=0.1] (4,3,0) -- (4,4,-1) -- (-1,9,-1) -- (-1,8,0);
    \scoped[on background layer]
    \fill[green,top color=green, bottom color=green, fill opacity=0.1] (0,0,7) -- (-1,1,7) -- (-1,8,0) -- (0,7,0); 
    
    \foreach \i in {1,2,3}
    {
    \filldraw [black] (\i,0,0) circle (.75pt);
    }
    
    \foreach \i in {1,2,3,4}
    {
    \filldraw [black] (0,\i,0) circle (.75pt);
    }
    
    \foreach \i in {1,2,3,4,5,6}
    {
    \filldraw [black] (0,0,\i) circle (.75pt);
    }

    \filldraw [black] (15,0,0) circle (1pt);
    \filldraw [black] (19,0,0) circle (0pt) node[below] {$y$};
    \filldraw [black] (15,5,0) circle (0pt) node[left] {$z$};
    \filldraw [black] (15,0,7) circle (0pt) node[below] {$x$};
    \draw[densely dotted,->] (15,0,0) -- (19,0,0);
    \draw[densely dotted,->] (15,0,0) -- (15,0,7);
    \draw[densely dotted,->] (15,0,0) -- (15,5,0);

    \filldraw [black] (18,4,0) circle (.75pt) node[right] {$(0,3,4)$};
    \filldraw [black] (15,4,3) circle (.75pt) node[left] {$(3,0,4)$};
    \filldraw [black] (15,1,6) circle (.75pt) node[left] {$(6,0,1)$};
    \filldraw [black] (16,0,6) circle (.75pt) node[below] {$(6,1,0)$};
    \filldraw [black] (18,0,4) circle (.75pt) node[right] {$(4,3,0)$};
    
    \fill[blue,top color=blue, bottom color=blue, fill opacity=0.1] (18,4,0)--(18,0,4)--(16,0,6)--(15,1,6)--(15,4,3);

    \foreach \i in {16,17,18}
    {
    \filldraw [black] (\i,0,0) circle (.75pt);
    }
    
    \foreach \i in {1,2,3,4}
    {
    \filldraw [black] (15,\i,0) circle (.75pt);
    }
    
    \foreach \i in {1,2,3,4,5,6}
    {
    \filldraw [black] (15,0,\i) circle (.75pt);
    }
\end{tikzpicture}\caption{$\mathscr{R}_{7,(6,3,4)}$}\label{fig:1}\end{figure}
\end{ej}

We will sometimes restrict to the case in which $0 < k < c_1+\cdots+c_n$; under this assumption, the Ehrhart polynomial has degree $n-1$. It is straightforward to prove that slices of prisms are always integral polytopes.

As we mentioned in the Introduction, the Ehrhart theory of the ``thin'' slices $\mathscr{R}_{k,\mathbf{c}}$ defined in \eqref{eq:defi-prism} and the ``fat'' slices $\mathscr{R}'_{a,b,\mathbf{c}}$ defined in \eqref{eq:fat-slice} is essentially the same. Let us state this more precisely. Two lattice polytopes $\mathscr{P}_1\subseteq \R^n$ and $\mathscr{P}_2\subseteq \R^m$ are said to be \emph{integrally equivalent} when there is an affine map $\varphi:\R^n\to\R^m$ such that its restriction to $\mathscr{P}_1$ induces a bijection $\varphi:\mathscr{P}_1\to\mathscr{P}_2$ which preserves the lattice, i.e. the image under $\varphi$ of $\Z^n \cap \operatorname{aff}\mathscr{P}_1$ is $\Z^m \cap \operatorname{aff}\mathscr{P}_2$, where $\operatorname{aff}\mathscr{P}_1$ denotes the affine space spanned by $\mathscr{P}_1$ and analogously for $\mathscr{P}_2$.
Something immediate from the definitions is that integrally equivalent polytopes have the same Ehrhart polynomial. 

\begin{prop}\label{prop:fat-thin}
    Let $\mathbf{c}=(c_1,\ldots,c_n)\in \mathbb{Z}_{>0}^n$ and $0\leq a < b$. Then
        \[ \ehr(\mathscr{R}'_{a,b,\mathbf{c}}, t) = \ehr(\mathscr{R}_{b,\mathbf{c}'},t),\]
    where $\mathbf{c}'\in \mathbb{Z}_{>0}^{n+1}$ is given by $\mathbf{c}' = (\mathbf{c}, b-a)$.
\end{prop}

\begin{proof}
    Let us prove that, using the notation of the statement, the two polytopes $\mathscr{R}'_{a,b,\mathbf{c}}$ and $\mathscr{R}_{b,\mathbf{c}'}$ are integrally equivalent. Notice that the polytope $\mathscr{R}_{b,\mathbf{c}'}$ lies in $\mathbb{R}^{n+1}$, whereas $\mathscr{R}'_{a,b,\mathbf{c}}\subseteq \mathbb{R}^n$. Let us consider the map $\pi : \mathbb{R}^{n+1}\to \mathbb{R}^n$ that forgets the last coordinate. The image of the polytope $\mathscr{R}_{b,\mathbf{c}'}\subseteq\mathbb{R}^{n+1}$ is given by
        \[ \pi(\mathscr{R}_{b,\mathbf{c}'}) = \left\{ x\in \mathscr{R}_{\mathbf{c}} : a\leq \sum_{i=1}^{n} x_i \leq b  \right\} = \mathscr{R}'_{a,b,\mathbf{c}}.\]
    Moreover, the restriction of $\pi$ to $\mathscr{R}_{b,\mathbf{c}'}$ is in fact a bijection and provides an integral equivalence.
\end{proof}

It is because of this property that we will focus only on the Ehrhart theory of the ``thin'' slices of prisms. 

\subsection{Polymatroids with the strong exchange property}

In \cite{herzog-hibi} Herzog and Hibi introduced the notion of \emph{polymatroid of Veronese type}. For the necessary background on polymatroids we refer to their article, in particular \cite[Example~2.6]{herzog-hibi}. 

\begin{prop}
    Base polytopes of discrete polymatroids of Veronese type are slices of prisms. Conversely, every slice of a prism arises in this way.
\end{prop}

In particular, since base polytopes of polymatroids are generalized permutohedra, we obtain that the edge directions of any slice of a prism is of the form $e_i-e_j$. On the other hand, in light of the form of the inequalities that describe a slice of a prism, it follows that they are alcoved polytopes (as in \cite{lam-postnikov-polypositroids}), hence:

\begin{cor}
    Slices of prisms are polypositroids.
\end{cor}

Undertaking the discussion on polymatroids of Veronese type, as we mentioned in the introduction, they admit a very neat characterization. They are (up to an affinity) precisely the discrete polymatroids that satisfy the strong exchange property (we refer to \cite{herzog-hibi-vladoiu} for the details).

\subsection{The Ehrhart polynomial explicitly}

Towards a proof of Theorem~\ref{main}, the first step is to give a somewhat explicit expression for the Ehrhart polynomial of $\mathscr{R}_{k,\mathbf{c}}$. To simplify the statement, we introduce some notation. From now on, whenever $\mathbf{c}=(c_1,\ldots,c_n)\in \mathbb{Z}_{> 0}^n$ is fixed, for each $0\leq j \leq n$ we will denote 
    \begin{equation}\label{eq:def-rho} \rho_{\mathbf{c},j}(s) := \#\left\{ I\in \binom{[n]}{j} : \sum_{i \in I} c_i = s\right\}.\end{equation}
In other words $\rho_{\mathbf{c},j}(s)$ is the number of ways of choosing exactly $j$ of the $c_i$'s in such a way that their sum is exactly $s$. If $j=0$ we define $\rho_{\mathbf{c},j}(s)$ to be $1$ if $s=0$ and $0$ otherwise.

\begin{teo}\label{thm:first-formula}
    Let $\mathbf{c}=(c_1,\dots,c_n)\in \mathbb{Z}_{>0}^n$. For each positive integer $k$, the Ehrhart polynomial of the polytope $\mathscr{R}_{k,\mathbf{c}}$ is given by
    \[ \ehr(\mathscr{R}_{k,\mathbf{c}},t)= \sum_{j=0}^{k-1}(-1)^j\sum_{v=0}^{k-1} \binom{t(k-v)+n-1-j}{n-1} \rho_{\mathbf{c},j}(v). \]
\end{teo}

\begin{proof}
    By using the definition, we can write
    \[ \mathscr{R}_{k,\mathbf{c}} = \left\{ x\in \mathbb{R}^n_{\geq 0} : \sum_{i=1}^n x_i = k \text{ and } x_i \leq c_i \text{ for all $1\leq i \leq n$} \right\}.\]
    Therefore, it follows that
    \begin{align*}
        \ehr(\mathscr{R}_{k,\mathbf{c}}, t) &= \#(t \mathscr{R}_{k,\mathbf{c}} \cap \mathbb{Z}^n)\\
        &= \# \left\{ x\in \mathbb{Z}_{\geq 0}^n: \sum_{i=1}^n x_i = kt \text{ and } x_i \leq c_it \text{ for all $1\leq i \leq n$} \right\}\\
        &= [x^{kt}] \prod_{i=1}^n (1+x+x^2+\cdots+x^{c_it}).
    \end{align*}
    Using the identity $1+x+\cdots+x^{c_it} = \dfrac{1-x^{c_it+1}}{1-x}$, we can further reduce
    \begin{equation}\label{eq:ehr_r_gf}
        \ehr(\mathscr{R}_{k,\mathbf{c}},t) = [x^{kt}] \left( \frac{1}{(1-x)^n} \prod_{i=1}^n (1-x^{c_it+1})\right). 
    \end{equation}
    Now, let us consider a generic coefficient of the factor $\prod_{i=1}^n (1-x^{c_it+1})$; we can write
        \begin{equation}\label{eq:coeff_of_product}
        [x^u] \prod_{i=1}^n (1-x^{c_it+1}) =  \sum_{j=0}^n(-1)^j\#\left\{i_1<\cdots<i_j : \sum_{\ell=1}^j (c_{i_\ell}t+1) = u\right\}.
        \end{equation}
    Thus, we can use what we obtained in equation \eqref{eq:coeff_of_product} to simplify our formula in \eqref{eq:ehr_r_gf}. Recall the classical generating function identity: $\frac{1}{(1-x)^n} = \sum_{i=0}^{\infty} \binom{n-1+i}{n-1} x^i$.
        \begin{align*}
            \ehr(\mathscr{R}_{k,\mathbf{c}},t) &= [x^{kt}] \left( \frac{1}{(1-x)^n} \prod_{i=1}^n (1-x^{c_it+1})\right)\\
            &= [x^{kt}] \left( \sum_{i=0}^{\infty} \binom{n-1+i}{n-1}x^i \cdot \prod_{i=1}^n (1-x^{c_it+1})\right)\\ 
            &= \sum_{u=0}^{kt} \binom{n-1+kt-u}{n-1}\left([x^u] \prod_{i=1}^n (1-x^{c_it+1})\right)\\
            &= \sum_{u=0}^{kt} \binom{n-1+kt-u}{n-1} \sum_{j=0}^{k-1}(-1)^j\#\left\{i_1<\cdots<i_j : \sum_{\ell=1}^j (c_{i_\ell}t+1) = u\right\}\\
            &= \sum_{j=0}^{k-1} (-1)^j \underbrace{\sum_{u=0}^{kt} \binom{n-1+kt-u}{n-1}\#\left\{i_1<\cdots<i_j : \sum_{\ell=1}^j (c_{i_\ell}t+1) = u\right\}}_{(\star)}.
        \end{align*}
    Let us focus on the sum labeled by $(\star)$. If $j=0$, then it reduces to $\binom{n-1+kt}{n-1}$, since only $u=0$ contributes, as the sum of the ``empty choice'' of $c_i$'s is zero by definition. On the other hand, if we assume that $j\geq 1$, then the condition $\sum_{\ell=1}^j (c_{i_\ell}t+1) = u$ implies that $(c_{i_1}+\cdots+c_{i_j})t + j = u$, hence $u\equiv j\pmod{t}$. Therefore, since $0\leq u\leq kt$, we look only at the values $u = j$, $u=t+j$, $u=2t+j$, \ldots, $u = (k-1)t+j$. Namely
    \begin{align*}
        (\star) &= \sum_{v=0}^{k-1} \binom{n-1+kt-(vt+j)}{n-1} \# \left\{i_1<\cdots<i_j : c_{i_1}+\cdots+c_{i_j} = v\right\}\\
        &= \sum_{v=0}^{k-1} \binom{t(k-v)+n-1-j}{n-1} \rho_{\mathbf{c},j}(v).
    \end{align*}
    Thus, we obtain
    \[ \ehr(\mathscr{R}_{k,\mathbf{c}},t) = \sum_{j=0}^{k-1}(-1)^j \sum_{v=0}^{k-1}\binom{t(k-v)+n-1-j}{n-1}\rho_{\mathbf{c},j}(v), \]
    and the proof is complete.
\end{proof}

Observe that the particular case $\mathbf{c} = (1,\ldots,1)$ yields
    \[ \rho_{\mathbf{c},j}(v) = \left\{ \begin{matrix} \binom{n}{j} & \text{ if $v=j$}\\
    0 & \text{ if $v\neq j$}\end{matrix}\right.\]
and the formula reduces to 
    \[ \ehr(\Delta_{k,n}, t) = \ehr(\mathscr{R}_{k,(1,\ldots,1)},t)= \sum_{j=0}^{k-1}(-1)^j \binom{n}{j}\binom{t(k-j)+n-1-j}{n-1},\]
which is the known formula for the Ehrhart polynomial of the hypersimplex. Notice that even in this very particular case this sum is alternating in sign and that the variable $t$ appears inside a binomial coefficient, which is actually a polynomial with some negative coefficients when $j\geq 2$.

\section{Weighted permutations}\label{sec:3}

As the statement of Theorem~\ref{main} anticipates, the role played by the numbers $W(\ell,n,m+1,\mathbf{c})$ is fundamental. Its definition is slightly involved and cumbersome. The main motivation comes from the papers \cite{ferroni1} and \cite{hanelyetal}. Nevertheless, we let the reader know that this section is completely self-contained. The technical part is the proof of Theorem~\ref{thm:technical}, which provides an explicit (but rather complicated) formula for $W(\ell,n,m+1,\mathbf{c})$. This will play a crucial role in the proof of Theorem~\ref{main}. 

When dealing with permutations $\sigma\in\mathfrak{S}_n$ we will denote  $C(\sigma)$ the set of all its cycles. 

\subsection{Weighted Lah Numbers Revisited} 

First, recall that the \textit{Lah number} $L(n,m)$ is defined as the number of ways of partitioning $[n]$ into a set of exactly $m$ blocks, each of which is internally ordered. For example, if we partition the set $[4]$ into two blocks, then we distinguish between $\{(1,3),(2,4)\}$ and $\{(1,3),(4,2)\}$. Let us denote by $\mathscr{L}(n,m)$ the set of all such partitions. It is not difficult to show that $L(n,m)=|\mathscr{L}(n,m)| = \frac{n!}{m!}\binom{n-1}{m-1}$.

\begin{defi}[\cite{ferroni1}]
	Let $\pi\in \mathscr{L}(n,m)$. We define the \textit{weight of $\pi$} by
		\[ w(\pi) := \sum_{b\in\pi} w(b),\]
	where $w(b)$ is the number of elements in the block $b$ that are smaller (as positive integers) than the first element in $b$.
\end{defi}

As a quick example, consider the partition $\pi=\{(5,3,7),(6,2,4,1)\}\in \mathscr{L}(7,2)$, which has weight $w(\pi)=1+3=4$. 

\begin{defi}[\cite{ferroni1}]
	We define the \textit{weighted Lah Numbers} $W(\ell,n,m)$ as the number of partitions $\pi\in\mathscr{L}(n,m)$ such that $w(\pi) = \ell$. The set of all such $\pi$ is denoted by $\mathscr{W}(\ell,n,m)$.
\end{defi}

In \cite[Remark 3.11]{ferroni1} it is provided a linear recurrence that the weighted Lah numbers satisfy and that may be used to compute them.

The reason why they are relevant in this paper is because we are going to provide a vast generalization for them. Essentially, apart of the three parameters $\ell$, $n$ and $m$, we will introduce an additional input $\mathbf{c}$ consisting of $n$ integer numbers. When all the entries of the vector $\mathbf{c}$ are ones, we will recover the weighted Lah numbers.

\subsection{Weights for permutations}

In \cite{hanelyetal} a different approach to the weighted Lah numbers is provided. We will review that version of their definition here. One particular advantage that it has is that these numbers can be thought as the quantity of ``weighted permutations'' satisfying a certain property. This suggests how to provide the right generalization for the main result of this section.

\begin{defi}[\cite{hanelyetal}]
    Let $\sigma\in \mathfrak{S}_n$. A \emph{weight} for $\sigma$ is a map $w:C(\sigma) \to \mathbb{Z}_{\geq 0}$. The \emph{total weight} of $\sigma$ with respect to $w$, denoted by $w(\sigma)$, is defined by
        \[ w(\sigma) = \sum_{\mathfrak{c}\in C(\sigma)} w(\mathfrak{c}).\]
    The pair $(\sigma,w)$ will be called a \emph{weighted permutation}.
\end{defi}

\begin{ej}\label{ex:basic}
    Assume that for every $\sigma\in \mathfrak{S}_5$, we consider the weight induced by $w(\mathfrak{c}) = \max \{i : i\in \mathfrak{c}\}$. If we consider the permutation $(2,1,5,4,3)\in \mathfrak{S}_5$, written as the product of the cycles $(1\; 2)(3\; 5)(4)$, its total weight is $2 + 5 + 4 = 11$. 
\end{ej}

\begin{prop}\label{prop:bijection1}
    There is a bijection between $\mathscr{L}(n,m)$ and the set of all weighted permutations $(\sigma,w)$ where $\sigma\in \mathfrak{S}_n$ has exactly $m$ cycles and $w(\mathfrak{c}) < |\mathfrak{c}|$ for all $\mathfrak{c}\in C(\sigma)$.
\end{prop}

\begin{proof}
    The bijection is constructed explicitly as follows. Take $(\sigma,w)$ as in the statement. To each cycle $\mathfrak{c}\in C(\sigma)$ we assign a linearly ordered set $b$ consisting of the elements of $\mathfrak{c}$ as follows: choose the first element of $b$ to be the $(w(\mathfrak{c})+1)$-th smallest element in $\mathfrak{c}$ and complete the ordering of $b$ according to the order induced by the cycle $\mathfrak{c}$. For example, if
    \[ \sigma = (1\; 2\; 6)(3\; 5\; 7)(4\; 8),\]
    where $w((1\; 2\; 6)) = 1$, $w((3\; 5\; 7))=2$ and $w((4\; 8)) = 0$, then we take the partition
    \[ \pi = \{(2,6,1),(7,3,5),(4,8)\}.\] 
    It is straightforward to verify that this is indeed a bijection.
\end{proof}

\begin{rem}\label{rem:clue}
    As a consequence of the preceding result, if we fix $0\leq \ell \leq n-m$ and restrict ourselves to the stratum $\mathscr{W}(\ell,n,m) \hookrightarrow \mathscr{L}(n,m)$, we see that it is in one-to-one correspondence with the family of all the weighted permutations $(\sigma, w)$ such that $\sigma\in \mathfrak{S}_n$ has $m$ cycles, $w(\mathfrak{c}) < |\mathfrak{c}|$ for every $\mathfrak{c}\in C(\sigma)$ and $w(\sigma) = \ell$.
\end{rem}

\subsection{Compatibility}\label{CompatibilitySection}

The bijection mentioned in Remark~\ref{rem:clue} suggests the following definition. Essentially, we use a certain bound for the weight of the cycles that depends on the elements that the cycles contain and not on their lengths as in Proposition~\ref{prop:bijection1}. 

\begin{defi}
    Let $(\sigma, w)$ be a weighted permutation, $\sigma\in \mathfrak{S}_n$. Let us fix the vector $\mathbf{c}=(c_1,\ldots,c_n)\in \mathbb{Z}_{>0}^n$. We say that $(\sigma,w)$ is \emph{$\mathbf{c}$-compatible} if 
        \[ w(\mathfrak{c}) < \sum_{i\in \mathfrak{c}} c_i \]
    for every cycle $\mathfrak{c}\in C(\sigma)$.
\end{defi}

\begin{ej}
    Consider the permutation $\sigma = (2,1,5,4,3)\in \mathfrak{S}_5$ with the weight $w$ given in Example \ref{ex:basic}. If we consider the vector $\mathbf{c} = (8,10,7,8,8)$, then $(\sigma,w)$ is $\mathbf{c}$-compatible. However, if we take the vector $\mathbf{c}=(3,3,3,3,3)$, then $(\sigma,w)$ is not $\mathbf{c}$-compatible.
\end{ej}

Let us consider $\mathbf{c} = \mathbf{1} = (1,\ldots,1)\in \mathbb{Z}_{> 0}^n$. A weighted permutation $\sigma\in \mathfrak{S}_n$ is $\mathbf{1}$-compatible if and only if each cycle of $\sigma$ is assigned a number smaller than its length. In other words, by Proposition~\ref{prop:bijection1} we have that $\mathscr{L}(n,m)$ is in bijection with the set of all $\mathbf{1}$-compatible weighted permutations $(\sigma,w)$ such that $\sigma\in \mathfrak{S}_n$ has exactly $m$ cycles.

\begin{defi}\label{def:cwhlah}
    For $\mathbf{c}\in\mathbb{Z}_{> 0}^n$, let us denote by $\mathscr{L}(n,m,\mathbf{c})$ the set of all $\mathbf{c}$-compatible weighted permutations $(\sigma,w)$ where $\sigma\in \mathfrak{S}_n$ has $m$ cycles. Let us denote by $\mathscr{W}(\ell,n,m,\mathbf{c})$ the family of all $(\sigma,w)\in \mathscr{L}(n,m,\mathbf{c})$ with total weight $w(\sigma) = \ell$. 
\end{defi}

If we use the notation
    \begin{align*}
        L(n,m,\mathbf{c}) &:= |\mathscr{L}(n,m,\mathbf{c})|,\\
        W(\ell,n,m,\mathbf{c}) &:= |\mathscr{W}(\ell,n,m,\mathbf{c})|,
    \end{align*}
then by taking $\mathbf{c} = \mathbf{1}$, we have that $L(n,m,\mathbf{1}) = L(n,m)$ and $W(\ell,n,m,\mathbf{1}) = W(\ell,n,m)$.

\subsection{A counting formula for weighted permutations}

The main goal now is to provide a formula for $W(\ell,n,m,\mathbf{c})$ that is as concrete as possible. We introduce some notation that will be useful not only in this section but later in the paper. 

If $[a,b]$ is an interval of (possibly negative) integers, we will denote 
    \[ P_{a,b}^s := \sum_{I\in \binom{[a,b]}{s}} \prod_{i \in I} i = \sum_{a\leq i_1 < \cdots < i_s \leq b} i_1\cdots i_s.\]
This is the $s$-th elementary symmetric polynomial in $b-a+1$ variables, evaluated in the integers of the interval $[a,b]$. 

\begin{rem}
    Observe that when $a=1$, the number $P_{a,b}^s$ reduces to a Stirling number of the first kind,
        \[ P_{1,b}^s = \stirling{b+1}{b+1-s},\]
    that is, the number of permutations in $\mathfrak{S}_{b+1}$ with $b+1-s$ cycles.
    Also, if $a > 0$ and $b > 0$, we have
    \begin{align} 
        P_{-a,b}^s &= \sum_{j=0}^s P_{-a,-1}^j P_{1,b}^{s-j}\nonumber\\
        &= \sum_{j=0}^s (-1)^j P_{1,a}^j P_{1,b}^{s-j}\nonumber\\
        &= \sum_{j=0}^s (-1)^j \stirling{a+1}{a+1-j}\stirling{b+1}{b+1-s+j}.\label{eq:P-stirling}
    \end{align}
\end{rem}

\begin{teo}\label{thm:technical}
 For every $\mathbf{c}\in \mathbb{Z}_{> 0}^n$, the following formula holds:
    \[
    W(\ell,n,m+1,\mathbf{c}) = \sum_{j=0}^{n}(-1)^jP^{n-1-m}_{-j+1,n-1-j}\sum_{i=0}^\ell\rho_{\mathbf{c},j}(i)\binom{m+\ell-i}{m}.
    \]
\end{teo}

\begin{proof}
    First, by using \eqref{eq:P-stirling} we can rewrite the expression on the right as follows
    \begin{align*}
        &\sum_{j=0}^{n}\sum_{u=0}^j(-1)^{j-u} \stirling{j}{j-u}\stirling{n-j}{m+1+u-j}\sum_{i=0}^\ell\rho_{\mathbf{c},j}(i)\binom{m+\ell-i}{m}\\
        &=\sum_{j=0}^{n}\sum_{u=0}^j(-1)^{j-u} \stirling{j}{j-u}\stirling{n-j}{m+1+u-j}\sum_{i=0}^\ell\sum_{\substack{A\in \binom{[n]}{j}\\ \sum_{a\in A} c_a=i}}\binom{m+\ell-i}{m}\\
        &=\sum_{j=0}^{n}\sum_{u=0}^j\sum_{i=0}^\ell\sum_{\substack{A\in \binom{[n]}{j}\\ \sum_{a\in A} c_a=i}}(-1)^{j-u} \stirling{j}{j-u}\stirling{n-j}{m+1+u-j}\binom{m+\ell-i}{m}.
    \end{align*}
    If we look at a fixed $A\in \binom{[n]}{j}$ such that $\sum_{a\in A}c_a=i$, the quantity
    \[\stirling{j}{j-u}\stirling{n-j}{m+1+u-j}\binom{m+\ell-i}{m}\] 
    can be seen as the number of weighted permutations $(\sigma,w)$ that simultaneously satisfy the following properties:
    \begin{itemize}
        \item $\sigma\in \mathfrak{S}_n$ has exactly $m+1$ cycles.
        \item $w(\sigma) =\ell$.
        \item Exactly $j-u$ of the cycles consist only of elements that are contained in $A$.
        \item The remaining $m+1-j+u$ cycles consist only of elements that are contained in $[n]\smallsetminus A$.
        \item $w(\mathfrak{c}) \geq \sum_{h \in \mathfrak{c}} c_h$ for each of the $j-u$ cycles $\mathfrak{c}\in C(\sigma)$ that consist only of elements contained in $A$. 
    \end{itemize}
    The last condition explains the factor $\binom{m+\ell-i}{m}$, which is exactly the number of ways of putting $\ell-i$ balls into $m+1$ boxes. This is because we put at least $i = \sum_{a\in A} c_a$ weight in the cycles consisting of elements in $A$, and then we can assign the remaining weight $\ell - i$ in $\binom{m+\ell-i}{m}$ different ways. 
    
    Let us call $\widehat{\mathscr{W}}(n,m+1,\ell,A,j-u,\mathbf{c})$ the set of all weighted permutations $(\sigma, w)$ satisfying the five conditions above (recall that $i$ is determined from $A$ and $\mathbf{c}$). Further, notice that \[\widehat{\mathscr{W}}(\ell,n,m+1,\varnothing,0,\mathbf{c}) = \mathscr{W}(\ell,n,m+1,\mathbf{c}),\]
    which is a direct consequence of the definitions.
    
    The statement that we want to prove is therefore equivalent to showing that
    \[ |\mathscr{W}(\ell,n,m+1,\mathbf{c})| =\sum_{j=0}^{n}\sum_{u=0}^j\sum_{i=0}^\ell\sum_{\substack{A\in \binom{[n]}{j}\\ \sum_{a\in A} c_a=i}}(-1)^{j-u} |\widehat{\mathscr{W}}(\ell,n,m+1,A,j-u,\mathbf{c})|.\]
    
    This equality is rather a consequence of an inclusion-exclusion argument that we explain now. In what follows, whenever $\mathscr{S}$ is a set, we use the notation $\mathscr{S}[x]$ to denote $0$ or $1$ according to whether $x\notin \mathscr{S}$ or $x\in \mathscr{S}$ respectively. 
    
    If $(\sigma,w)\in \mathscr{W}(\ell,n,m+1,\mathbf{c})$, then $(\sigma,w)$ is present only in the set $\widehat{\mathscr{W}}(\ell,n,m+1,\varnothing,0,\mathbf{c})$, which corresponds to the case in which $u = j = 0$, and appears with a plus sign.
    
    On the contrary, let us fix some weighted permutation in $\mathfrak{S}_n$, with $m+1$ cycles and total weight $\ell$ but is not $\mathbf{c}$-compatible; in other words, $(\sigma,w)\notin \mathscr{W}(\ell,n,m+1,\mathbf{c})$. Let $\mathfrak{c}_1,\dots,\mathfrak{c}_b$ be the cycles  $\mathfrak{c}\in C(\sigma)$ for which $w(\mathfrak{c})\geq \sum_{h\in \mathfrak{c}} c_h$. For $B\subseteq [b]$, let $A_B:= \bigcup_{s\in B} \mathfrak{c}_s$ (where we are regarding the cycles $\mathfrak{c}_s$ as sets). Then $(\sigma,w)$ is contained in precisely the sets $\widehat{\mathscr{W}}(\ell,n,m+1,A_B,|B|,\mathbf{c})$ for $B\subseteq [b]$. Therefore,
    \begin{align*}
        &\sum_{j=0}^{n}\sum_{u=0}^j\sum_{i=0}^\ell\sum_{\substack{A\in \binom{[n]}{j}\\ \sum_{a\in A} c_a=i}}(-1)^{j-u}\cdot \widehat{\mathscr{W}}(\ell,n,m+1,A,j-u,\mathbf{c})[(\sigma,w)]\\
        &=\sum_{B\subseteq [b]} (-1)^{|B|}\cdot \widehat{\mathscr{W}}(n,m+1,\ell,A_B,|B|,\mathbf{c})[(\sigma,w)]\\
        &=\sum_{B\subseteq  [b]}(-1)^{|B|}=0.
    \end{align*}
    Hence, the proof is complete.
\end{proof}

\section{Ehrhart coefficients}\label{sec:4}

In this section we will prove Theorem~\ref{main} and then comment briefly about a conjecture that arose in the study of these polynomials.

\subsection{The proof}

We start with a Lemma that will be used later.

\begin{lem}\label{lem:coeff-of-binomial}
    The following equality holds
        \[ [t^m] \binom{t(k-v)+n-1-j}{n-1} =\frac{1}{(n-1)!} \, (k-v)^m P_{-j+1,n-1-j}^{n-1-m}. \]
\end{lem}

\begin{proof}
    By definition we have that
    \[\binom{t(k-v)+n-1-j}{n-1} = \frac{1}{(n-1)!} \prod_{s=1}^{n-1} \left(t(k-v)+n-j-s\right).\]
    Therefore, apart from the factor $\frac{1}{(n-1)!}$, we see that the coefficient of degree $m$ consists of the product between $(k-v)^m$ and the sum of all products of $n-m-1$ numbers chosen in the interval of integers $[-j+1, n-1-j]$, namely $P_{-j+1,n-1-j}^{n-1-m}$. 
\end{proof}

We are now ready to prove Theorem~\ref{main}.

\begin{teo}
    Let $\mathbf{c}=(c_1,\dots,c_n)\in \mathbb{Z}_{>0}^n$. Let $0\leq m\leq n-1$. The coefficient of $t^m$ in $\ehr(\mathscr{R}_{k,\mathbf{c}},t)$ is given by
    \[
    [t^m]\ehr(\mathscr{R}_{k,\mathbf{c}},t) = \frac{1}{(n-1)!}\sum_{\ell = 0}^{k-1}W(\ell,n,m+1,\mathbf{c})A(m,k-\ell-1).
    \]
\end{teo}

\begin{proof}
    By using the formula that we obtained in Theorem~\ref{thm:first-formula} and Lemma~\ref{lem:coeff-of-binomial}, we have the following chain of equalities:
    \begin{align*}
        [t^m]\ehr(\mathscr{R}_{k,\mathbf{c}},t) &= \sum_{j=0}^{k-1}(-1)^j\sum_{v=0}^{k-1} [t^m]\binom{t(k-v)+n-1-j}{n-1} \rho_{\mathbf{c},j}(v)\\
        &=\frac{1}{(n-1)!}\sum_{j=0}^{k-1}\sum_{v=0}^{k-1}(-1)^{j} \,(k-v)^m\, P^{n-1-m}_{-j+1,n-1-j}\, \rho_{\mathbf{c},j}(v)\\
        &=\frac{1}{(n-1)!}\sum_{v=0}^k (k-v)^m\, \sum_{j=0}^{k-1}(-1)^j\, P^{n-1-m}_{-j+1,n-1-j} \,\rho_{\mathbf{c},j}(v)\\
        &=\frac{1}{(n-1)!}\sum_{v=0}^k (k-v)^m\, \sum_{j=0}^{n}(-1)^j\, P^{n-1-m}_{-j+1,n-1-j} \,\rho_{\mathbf{c},j}(v),
    \end{align*}
    where in the second to last step we changed the summation order and the upper limit for $v$ because for $v=k$ we are just adding a zero due to the factor $(k-v)^m$, and in the last step we changed the upper limit for $j$, as $\rho_{\mathbf{c},j}(v)$ is zero whenever $j\geq k$, as $v$ is always at most $k$ and the $c_i$'s are strictly positive. Let us introduce the following notation:
    \begin{align*}
        F_{n,m}(x) &:= \sum_{v=0}^{\infty} \left( \sum_{j=0}^{n}(-1)^j P^{n-1-m}_{-j+1,n-1-j} \rho_{\mathbf{c},j}(v)\right) x^v,\\
        G_{m}(x) &:= \sum_{v = 0}^{\infty} v^m \cdot x^v.
    \end{align*}
    Using these names, the preceding chain of equalities reduces to
    \begin{equation}\label{eq:help1} [t^m]\ehr(\mathscr{R}_{k,\mathbf{c}},t) = \frac{1}{(n-1)!}[x^k] \left(F_{n,m}(x)\cdot G_{m}(x)\right).\end{equation}
    
    It is well known that $G_m(x)$ satisfies the following identity:
    \begin{equation}\label{eq:help2}
    G_m(x) = \frac{1}{(1-x)^{m+1}} \sum_{i=0}^{m} A(m,i) x^{i+1},\end{equation}
    where $A(m,i)$ is an Eulerian number (see \cite[p. 40]{stanley-combinatorics1} for instance). On the other hand, notice that
    \begin{align}
        \frac{1}{(1-x)^{m+1}} F_{n,m}(x) &= \sum_{\ell=0}^{\infty} \left(\sum_{i=0}^{\ell}  \sum_{j=0}^{n}(-1)^j P^{n-1-m}_{-j+1,n-1-j} \rho_{\mathbf{c},j}(i) \binom{m+\ell-i}{m}\right) x^{\ell}\nonumber \\
        &= \sum_{\ell = 0}^{\infty} W(\ell,n,m+1,\mathbf{c}) x^{\ell}\label{eq:help3},
    \end{align}
    where in the last step we used Theorem~\ref{thm:first-formula}. Now the equality of the statement follows from equations \eqref{eq:help1}, \eqref{eq:help2} and \eqref{eq:help3}, just by writing
    \begin{align*} [t^m]\ehr(\mathscr{R}_{k,\mathbf{c}},t) &=\frac{1}{(n-1)!} [x^k] \left(\sum_{\ell=0}^{\infty} W(\ell,n,m+1,\mathbf{c}) x^{\ell} \cdot \sum_{i=1}^m A(m,i-1) x^i\right)\\
    &= \frac{1}{(n-1)!}\sum_{\ell = 0}^{k-1}W(\ell,n,m+1,\mathbf{c})A(m,k-\ell-1).\qedhere
    \end{align*}
\end{proof}

\begin{rem}
    If we change the last coordinate of $\mathbf{c}=(c_1,\ldots,c_{n-1},c_n)$ by $\mathbf{c}'=(c_1,\ldots,c_{n-1},c_n+1)$, it is immediate by definition that $\mathscr{W}(\ell,n,m+1,\mathbf{c})\subseteq \mathscr{W}(\ell,n,m+1,\mathbf{c}')$. More generally, by induction it follows that whenever the vector $\mathbf{c}'-\mathbf{c}$ has nonnegative coefficients, one has that \[\mathscr{W}(\ell,n,m+1,\mathbf{c}) \subseteq \mathscr{W}(\ell,n,m+1,\mathbf{c}').\]
    Due to the preceding result, this monotonicity property has a counterpart for the Ehrhart polynomials.
\end{rem}

\begin{cor}
    If $\mathbf{c},\mathbf{c}'\in \mathbb{Z}^n_{>0}$ are vectors such that $\mathbf{c}'-\mathbf{c}\in \mathbb{Z}^n_{\geq 0}$, then
        \[ \ehr(\mathscr{R}_{k,\mathbf{c}},t) \preceq \ehr(\mathscr{R}_{k,\mathbf{c}'},t),\]
    where $\preceq$ denotes coefficient-wise inequality.
\end{cor}

\subsection{A unit-circle-rootedness conjecture}

While verifying computationally the main results of the present paper, the following intriguing problem, that we pose here as a conjecture, arose.

\begin{conj}
    For each $\mathbf{c}=(c_1,\ldots,c_n)\in\mathbb{Z}_{>0}^n$ and each $0\leq m\leq n - 1$, the polynomial defined by
        \[ p_{n,m,\mathbf{c}}(x) = \sum_{\ell = 0}^{\infty} W(\ell,n,m+1,\mathbf{c})\, x^{\ell},\]
    has all of its complex roots lying on the unit circle $|z| = 1$.
\end{conj}

Notice that although we used the upper limit $\infty$ in the sum, it actually yields a polynomial, as for $\ell > c_1+\cdots+c_n$, one will certainly have $W(\ell,n,m+1,\mathbf{c}) = 0$. A reasonable question is whether some techniques regarding $h^*$-polynomials having roots on the unit circle can be applied to these particular polynomials (notice that they may not be the $h^*$-polynomial of a polytope, as they can have a constant term larger than $1$). See \cite{braun-liu} for results regarding polytopes having a unit-circle-rooted $h^*$-polynomial.

\section{The \texorpdfstring{$h^*$}{h*}-polynomial}\label{sec:5}

Recall that the $h^*$-polynomial of a lattice polytope $\mathscr{P}$ of dimension $d$ is defined as the numerator of the generating function of its Ehrhart polynomial. Namely, it is defined as the only polynomial satisfying
    \[ \sum_{j=0}^{\infty} \ehr(\mathscr{P}, j) \, x^j = \frac{h^*(\mathscr{P},x)}{(1-x)^{d+1}}.\]

A classical result due to Stanley \cite{stanley-hstar} shows that $h^*(\mathscr{P},x)$ is a polynomial with nonnegative integer coefficients and has degree at most $d$.  A basic property of $h^*$-polynomials is that $h^*(\mathscr{P},1)$ is equal to the normalized volume of the polytope $\mathscr{P}$.

In \cite{early} Early conjectured a combinatorial formula for the coefficients of the $h^*$-polynomial of the hypersimplex and, more generally, for the $h^*$-polynomial of the polytopes of the form $\mathscr{R}_{k,\mathbf{c}}$ where $\mathbf{c}=(r,\ldots,r)\in\mathbb{Z}_{>0}^n$. In \cite{kim} Kim provided a proof of Early's conjectures. The aim of this section is to provide the right generalization of the conjecture posed by Early for \emph{all} the polytopes of the form $\mathscr{R}_{k,\mathbf{c}}$ and then outline a proof that extends Kim's procedure.

\subsection{A review of terminology}

In the same way that the $\mathbf{c}$-compatible weighted permutations play a key role in the description of the coefficients of the Ehrhart polynomial of $\mathscr{R}_{k,\mathbf{c}}$, for the $h^*$-polynomial we have to introduce a different object, already studied in both \cite{early} and \cite{kim} and endow it with an extra condition that we call again ``$\mathbf{c}$-compatibility''.

Recall that a \emph{cyclically ordered partition} of $[n]$ is a partition of $[n]$ into disjoint blocks that are ordered cyclically. For example, the partitions of $[5]$ given by $(\{1,2\}, \{3,5\}, \{4\})$ and $(\{3,5\}, \{4\}, \{1,2\})$ are considered as equal. The set of blocks of such a partition $\xi$ will be customarily denoted by $P(\xi)$. 


\begin{defi}
    A \emph{decorated ordered set partition} of type $(k,n)$ consists of a cyclically ordered partition $\xi$ of $[n]$ and a function $w:P(\xi) \to \mathbb{Z}_{\geq 0}$ such that 
        \[ \sum_{\mathfrak{p}\in P(\xi)} w(\mathfrak{p}) = k.\]
        
    For a vector $\mathbf{c}=(c_1,\ldots,c_n)\in\mathbb{Z}_{>0}^n$, we say that a decorated ordered set partition $\xi$ is \emph{$\mathbf{c}$-compatible} if
        \[ w(\mathfrak{p}) < \sum_{i\in\mathfrak{p}} c_i\]
    for all $\mathfrak{p}\in P(\xi)$.
\end{defi}
A decorated ordered set partition of type $(k,n)$, $\xi=(L_1,\ldots,L_m)$, can be represented as a set of $k$ points in a circle ordered in a clockwise fashion. The blocks $L_i$ are placed among these points in such a way that the clockwise distance between $L_i$ and $L_{i+1}$ is $w(L_i)$. Observe that this only depends on the cyclic order and not in the particular choice of the first block $L_1$.

Notice that in the case in which $\mathbf{c}=(1,\ldots,1)$ the notion of being $\mathbf{c}$-compatible was named as ``hypersimplicial'' in \cite{early}, whereas in the case $\mathbf{c}=(r,\ldots,r)$ it was called ``$r$-hypersimplicial'' in \cite{kim}.

\begin{defi}
    The \emph{winding number} of a decorated ordered set partition of type $(k,n)$ is defined as the only number $d$ such that
        \[ dk = \lambda_1 + \cdots + \lambda_n,\]
    where $\lambda_i$ denotes the clockwise distance between the block containing $i$ and the block containing $i+1$ (modulo $n$).
\end{defi}

Having set all these names and notations, we are ready to state the main result of this section. The proof is carried out in the next subsection.

\begin{teo}
    Let $\mathbf{c}=(c_1,\ldots,c_n)\in\mathbb{Z}_{>0}^n$ and $0<k<c_1+\cdots+c_n$. For each $0\leq i\leq n$, the coefficient $[x^i]h^*(\mathscr{R}_{k,\mathbf{c}},x)$ equals the number of $\mathbf{c}$-compatible decorated ordered set partitions of type $(k,n)$ having winding number $i$.
\end{teo}

\subsection{A proof \`a la Kim}

We will outline a proof of Theorem~\ref{main-hstar}. Some details are omitted as the proofs are carried out verbatim from \cite{kim}; in particular, to improve readability we have decided to use the same (if not, very similar) notation. 

The family of all the partitions of a set $A$ will be henceforth denoted by $\Pi(A)$. Also, we will use the notation
    \[ \binom{n}{a}_b := [x^a] (1+x+x^2+\cdots+x^{b-1})^n,\]
so that in particular when $b=2$, we recover the classical binomial numbers. 

\begin{lem}\label{lemma:gen-binomial}
    The following formula holds
    \[ \binom{n}{a}_b = \sum_{j=0}^{\lfloor\frac{a}{b}\rfloor} (-1)^j \binom{n}{j}\binom{n-1+a-bj}{n-1}.\]
\end{lem}

\begin{proof}
    The proof is a standard argument with generating functions. 
    \begin{align*}
        \binom{n}{a}_b &= [x^a]\left(\sum_{i=0}^{b-1} x^i\right)^n\\
        &= [x^a]\frac{(1-x^b)^n}{(1-x)^n}\\
        &= [x^a]\left(\sum_{j=0}^n (-1)^j \binom{n}{j} x^{bj} \right)\left(\sum_{j=0}^{\infty}\binom{n-1+j}{n-1} x^j\right)\\
        &= \sum_{j=0}^{\lfloor\frac{a}{b}\rfloor} (-1)^j\binom{n}{j}\binom{n-1+a-bj}{n-1}.\qedhere
    \end{align*}
\end{proof}

\begin{defi}
    Let $\xi$ be a decorated ordered set partition of type $(k,n)$ and let $\mathbf{c}=(c_1,\ldots,c_n)\in \mathbb{Z}_{>0}^n$. The set of \emph{$\mathbf{c}$-bad blocks} of $\xi$ is the family
        \[ I_{\mathbf{c}}(\xi) = \left\{ \mathfrak{p}\in P(\xi) : w(\mathfrak{p}) \geq \sum_{i\in \mathfrak{p}} c_i\right\}.\]
    If $S$ is a family of pairwise disjoint subsets of $[n]$, we define
        \[ K_{\mathbf{c}}(S) := \{\xi : I_{\mathbf{c}}(\xi) \supseteq S\},\]
    and 
        \[H_{\mathbf{c}}(T) = \sum_{S\in \Pi(T)} (-1)^{|S|} |K_{\mathbf{c}}(S)|.\]
\end{defi}

\begin{prop}\label{kim2.11}
    The number of $\mathbf{c}$-compatible decorated ordered set partitions of type $(k,n)$ with winding number $d$ is given by 
    \[
    \sum_{T\subseteq [n]} H_{\mathbf{c}}(T).
    \]
\end{prop}

\begin{proof}
    The proof follows \emph{mutatis mutandis} from \cite[Proposition 2.11]{kim}.
\end{proof}

From now on, if we consider a cyclically ordered partition $\xi = (L_1,\ldots,L_m)$, we will consider the indices of the blocks modulo $m$. For instance $L_0=L_m$ and $L_{m+2} = L_{2}$.

\begin{defi}
    Let $T\subseteq [n]$ with $n\notin T$. 
    \begin{itemize}
        \item A \emph{$T$-singlet block} is a singleton $\{t\}\subseteq T$. 
        \item A sequence of $T$-singlet blocks $(L_i,\dots,L_{i+j})$ where $L_{i+u}=\{t_{i+u}\}$ occurring in a decorated ordered set partition $\xi$ is \emph{$\mathbf{c}$-packed} if $w(L_{i+u}) = c_{t_{i+u}}$ for all $0\leq u\leq j-1$ and $w(L_{i+j})\geq c_{t_{i+j}}$. 
        \item A $\mathbf{c}$-packed sequence is \emph{increasing $\mathbf{c}$-packed} if $t_i<t_{i+1}<\cdots<t_{i+j}$. 
        \item An increasing $\mathbf{c}$-packed sequence is \emph{maximal} if it is not contained in a larger increasing $\mathbf{c}$-packed sequence.
    \end{itemize}
\end{defi}

\begin{lem}\label{CorrespondenceLem}
Let $S=\{B_1,\dots,B_u\}\in \Pi(T)$ where $n\notin T$ and $T=\{t_1<\cdots<t_m\}$. Write the elements of $B_i$ in increasing order as $B_i=\{t_{j_1}<t_{j_2}<\cdots<t_{j_{u_i}}\}$. Then $K_{\mathbf{c}}(S)$ is in bijection with set of elements of $K_{\mathbf{c}}(\{\{t_1\},\dots,\{t_m\}\})$ that have an increasing $\mathbf{c}$-packed sequence $(\{t_{j_1}\},\{t_{j_2}\},\dots,\{t_{j_{u_i}}\})$ for all $i$.
\end{lem}

\begin{proof}
    See \cite[Lemma 2.14]{kim}.
\end{proof}

Fix $T=\{t_1<\cdots<t_m\}\subseteq [n]$ with $n\notin T$. For each $S\in \Pi(T)$, Lemma~\ref{CorrespondenceLem} provides a map
    \[
    i_S: K_{\mathbf{c}}(S) \hookrightarrow K_{\mathbf{c}}(\{t_1\},\ldots,\{t_m\}).
    \]
If we name $\chi_S$ the characteristic map of $i_S(K_{\mathbf{c}}(S))$, in other words,     \[ \chi_s(\xi) = \left\{
    \begin{matrix} 
    1 & \text{ if $\xi\in i_S(K_{\mathbf{c}}(S))$}\\
    0 & \text{ if $\xi\notin i_S(K_{\mathbf{c}}(S))$}
    \end{matrix}\right.,\]
then we have that
\begin{align*}
    H_{\mathbf{c}}(T) &= \sum_{S\in \Pi(T)}(-1)^{|S|} |K_{\mathbf{c}}(S)| = \sum_{S\in \Pi(T)}(-1)^{|S|} |i_S(K_{\mathbf{c}}(S))|\\
    &=\sum_{S\in \Pi(T)}(-1)^{|S|} \sum_{\xi\in K_{\mathbf{c}}(\{\{t_1\},\dots,\{t_m\}\})}\chi_S(\xi)\\
    &=\sum_{\xi\in K_{\mathbf{c}}(\{\{t_1\},\dots,\{t_m\}\})}\sum_{S\in \Pi(T)}(-1)^{|S|}\chi_S(\xi).
\end{align*}

\begin{prop}\label{chi-sum}
    Fix $T\subseteq [n]$ with $n\notin T$. If $\xi \in K_\mathbf{c}(T)$ does not have an increasing $\mathbf{c}$-packed sequence of length greater than 1, then the sum $\sum_{S\in \Pi(T)}(-1)^{|S|}\chi_S(\xi)$ is equal to $(-1)^{|T|}$. Otherwise, this sum is zero.
\end{prop}

\begin{proof}
    This can be obtained by a reasoning entirely analogous to that of \cite[Proposition 2.17]{kim}.
\end{proof}

Proposition~\ref{chi-sum} motivates the following definition.

\begin{defi}
    For a fixed $T=\{t_1<\cdots<t_m\}\subseteq [n]$ with $n\notin T$, define $\hat{K}_{\mathbf{c}}(T)$ to be the decorated ordered set partitions of $K_\mathbf{c}(T)$ that do not contain an increasing $\mathbf{c}$-packed sequence of length greater than $1$. 
\end{defi}

Notice by Proposition~\ref{chi-sum} and the short calculation preceding it, we have that
\[
H_\mathbf{c}(T)=(-1)^{|T|}\hat{K}_\mathbf{c}(T).
\]


\begin{prop}\label{kim2.22}
    The set $\hat{K}_\mathbf{c}(T)$ is in bijection with the following set
    \[
    \left\{(v_1,\ldots,v_n)\in \mathbb{Z}^n : \sum_{i=1}^n v_i = (k-c_T)d \text{ and }\;
    \begin{matrix}
        0&\leq v_i\leq& k-c_T-1 & \text{for each $i\notin T$}\\
        1&\leq v_i\leq &k-c_T & \text{for each $i\in T$}
    \end{matrix}\right\}
    \]
    where we use the notation $c_T:=\sum_{i\in T} c_i$.
\end{prop}

\begin{proof}
    See \cite[Proposition 2.22]{kim}.
\end{proof}

The idea is to provide a formula for $H_{\mathbf{c}}$. Since $H_{\mathbf{c}} = (-1)^{|T|} \hat{K}_{\mathbf{c}}(T)$, it is enough to enumerate the elements in the set of Proposition~\ref{kim2.22}. Its size can be calculated as follows. Let us consider new variables $v'_1,\ldots,v'_n$. For each $i\notin T$, we set $v_i'=v_i$, otherwise we put $v_i'=v_i-1$. Then the cardinality of the set in Proposition~\ref{kim2.22}
can be seen as the number of nonnegative integer solutions to $v'_1+\cdots+v'_n=(k-c_T)d-|T|$ under the constraints $0\leq v_i'\leq k-c_T-1$ for each $1\leq i\leq n$. This coincides with the coefficient of $x^{(k-c_T)d-|T|}$ in $(1+x+\cdots+x^{k-c_T-1})^n$, which is the number $\binom{n}{(k-c_T)d-|T|}_{k-c_T}$.

Putting all the pieces together, we obtain
\begin{align}
    \sum_{T\subseteq [n]}H_\mathbf{c}(T) &=\sum_{T\subseteq [n]}(-1)^{|T|}|\hat{K}_\mathbf{c}(T)|\nonumber\\
    &=\sum_{T\subseteq [n]}(-1)^{|T|}\binom{n}{(k-c_T)d-|T|}_{k-c_T}\nonumber\\
    &=\sum_{m=0}^n\sum_{|T|=m}(-1)^{|T|}\binom{n}{(k-c_T)d-|T|}_{k-c_T}\nonumber\\
    &=\sum_{m=0}^{k-1}\sum_{|T|=m}(-1)^{m}\binom{n}{(k-c_T)d-m}_{k-c_T}\nonumber\\
    &=\sum_{m=0}^{k-1}\sum_{v=0}^{k-1}\sum_{\substack{|T|=m\nonumber\\ c_T=v}}(-1)^{m}\binom{n}{(k-v)d-m}_{k-v}\nonumber\\
    &=\sum_{m=0}^{k-1}\sum_{v=0}^{k-1}(-1)^{m}\rho_{\mathbf{c},m}(v)\binom{n}{(k-v)d-m}_{k-v} \nonumber\\
    &=\sum_{m=0}^{k-1}\sum_{v=0}^{k-1}(-1)^{m}\rho_{\mathbf{c},m}(v)\sum_{i=0}^{d}(-1)^i\binom{n}{i}\binom{n-1+(k-v)(d-i)-m}{n-1}\label{eq:lemma-used}\\
    &=\sum_{i=0}^{d}(-1)^i \binom{n}{i} \left(\sum_{m=0}^{k-1}\sum_{v=0}^{k-1}(-1)^{m}\rho_{\mathbf{c},m}(v)\binom{n-1+(k-v)(d-i)-m}{n-1}\right)\nonumber\\
    &=\sum_{i=0}^{d}(-1)^i \binom{n}{i} \ehr(\mathscr{R}_{k,\mathbf{c}},d-i)\label{eq:i-used-ehr-formula}\\
    &= [x^d] \left((1-x)^n\, \sum_{j=0}^{\infty} \ehr(\mathscr{R}_{k,\mathbf{c}}, j) x^j\right)\nonumber\\
    &=[x^d] h^*(\mathscr{R}_{k,\mathbf{c}},x),\label{eq:i-used-definition}
\end{align}
where in \eqref{eq:lemma-used} we used Lemma~\ref{lemma:gen-binomial}, in \eqref{eq:i-used-ehr-formula} we leveraged the formula of Theorem~\ref{thm:first-formula} and in \eqref{eq:i-used-ehr-formula} we just used the definition of the $h^*$-polynomial. Now, combining this with Proposition~\ref{kim2.11}, we have proved Theorem~\ref{main-hstar}.

\subsection{A real-rootedness conjecture}

In the context of the study of the $h^*$-polynomials of the polytopes $\mathscr{R}_{k,\mathbf{c}}$ we pose the following conjecture.

\begin{conj}
    Let $\mathbf{c}=(c_1,\ldots,c_n)\in \mathbb{Z}_{>0}^n$ and $k>0$. Then $h^*(\mathscr{R}_{k,\mathbf{c}},x)$ is real-rooted. Moreover, if $c_n\geq 2$, define $\mathbf{c}'=(c_1,\ldots,c_{n-1},c_n-1,1)\in \mathbb{Z}^{n+1}_{>0}$, then the roots of $h^*(\mathscr{R}_{k,\mathbf{c}},x)$ and $h^*(\mathscr{R}_{k,\mathbf{c}'},x)$ interlace.
\end{conj}

Observe that this would automatically provide a proof for the unimodality of the coefficients of the $h^*$-polynomial of all hypersimplices, a long-standing open problem regarding the combinatorics of the hypersimplex, see \cite{braun}. In particular, if we use the notation $p(x)\lessdot q(x)$ to denote that $p$ and $q$ interlace, after fixing $k$ and $n$, our conjecture implies that one can produce the sequence of polynomials
    \[ h^*(\mathscr{R}_{k,(n)},x) \lessdot h^*(\mathscr{R}_{k,(n-1,1)},x)\lessdot h^*(\mathscr{R}_{k,(n-2,1,1)},x) \lessdot \cdots \lessdot h^*(\mathscr{R}_{k,(1,\ldots,1)},x),\]
and $h^*(\mathscr{R}_{k,(1,\ldots,1)},x)=h^*(\Delta_{k,n},x)$. It is reasonable to search for recurrences that these polynomials satisfy and use techniques as those used in \cite{haglund-zhang} or \cite{solus-gustafsson}.

\section{Some applications}\label{sec:6}

In this section we discuss some applications that are of independent interest. In some cases we omit details and definitions, and we advise to look at the references mentioned correspondingly.

\subsection{Volumes}

As we mentioned in the introduction, the leading coefficient of $\ehr(\mathscr{P},t)$ is $\vol(\mathscr{P})$. In particular, by using the formula of Theorem~\ref{main} for $m=n-1$ we obtain Corollary \ref{coromain}. We restate it and prove it now.

\begin{cor}
    Let $\mathbf{c}=(c_1,\dots,c_n)\in \mathbb{Z}_{>0}^n$ and $0 < k < c_1+\cdots+c_n$. Then the volume of $\mathscr{R}_{k,\mathbf{c}}$ is given by
        \[ \vol(\mathscr{R}_{k,\mathbf{c}}) = \frac{1}{(n-1)!} \sum_{\ell = 0}^{k-1} B(\ell,\mathbf{c})A(n-1,k-\ell-1)\]
    where $B(\ell,\mathbf{c})$ is defined as the number of ways of placing $\ell$ indistinguishable balls into $n$ boxes of capacities $c_1-1, \ldots, c_n-1$ respectively.
\end{cor}

\begin{proof}
    Given that we assume that $0 < k < c_1+\cdots+c_n$, the Ehrhart polynomial has degree $n-1$ and hence the volume is obtained by taking $m=n-1$ in the statement of our main theorem. 
    \begin{align*}
        \vol(\mathscr{R}_{k,\mathbf{c}}) &= [t^{n-1}] \ehr(\mathscr{R}_{k,\mathbf{c}},t)\\
        &= \frac{1}{(n-1)!} \sum_{\ell = 0}^{k-1} W(\ell, n,n, \mathbf{c}) A(n-1,k-\ell-1)
    \end{align*}
    Notice that we have a factor $W(\ell,n,n,\mathbf{c})$ in each summand. The only permutation $\sigma\in \mathfrak{S}_n$ with $n$ cycles is the identity. Therefore, in order for the permutation to admit a $\mathbf{c}$-compatible weight, we have to ensure that each of the numbers $i\in [n]$ is assigned a value between $0$ and $c_i-1$. To guarantee that the total weight is $\ell$, we can think that we are putting $\ell$ balls inside boxes with capacities $c_1-1,\ldots, c_n-1$, as is described in the statement.
\end{proof}

\subsection{Flag Eulerian numbers}

In this subsection we will extend and generalize  some of the results of \cite{han-josuat}. Let $\mathbf{c}=(c_1,\ldots,c_n)\in\mathbb{Z}^n_{>0}$. In what follows we will define objects that coincide with those defined in \cite{han-josuat} when $c_1=\cdots=c_n=r$. A \emph{$\mathbf{c}$-colored permutation on $[n]$} consists of  a permutation $\sigma\in\mathfrak{S}_n$ and a function $\mathbf{s}:[n]\to \mathbb{Z}_{\geq 0}$ such that $s_i:=\mathbf{s}(i) \leq c_i-1$ for each $i\in [n]$. The set of all $\mathbf{c}$-colored permutations for a fixed value of $n$ is denoted by $\mathfrak{S}_n^{(\mathbf{c})}$. The \emph{descent number} of a $\mathbf{c}$-colored permutation $(\sigma, \mathbf{s})\in \mathfrak{S}_n^{(\mathbf{c})}$ is defined as $\operatorname{des}(\sigma,\mathbf{s}) := \#\operatorname{Des}(\sigma,\mathbf{s})$ where
    \[ \operatorname{Des}(\sigma,\mathbf{s}) := \left\{ i\in [n-1] : s_i > s_{i+1}\, \text{ or $s_i=s_{i+1}$ and $\sigma_i>\sigma_{i+1}$}\right\}.\]
The \emph{flag descent number} of a $\mathbf{c}$-colored permutation $(\sigma,\mathbf{s})\in\mathfrak{S}_n^{(\mathbf{c})}$ is defined by
    \[ \operatorname{fdes}(\sigma,\mathbf{s}) := s_n+ \sum_{i\in \operatorname{Des}(\sigma,\mathbf{s})} c_{i+1}. \]
The \emph{flag Eulerian numbers} are defined by 
    \[A_{n,k}^{(\mathbf{c})} := \#\left\{ (\sigma,\mathbf{s}) \in \mathfrak{S}^{(\mathbf{c})}_n: \operatorname{fdes}(\sigma,\mathbf{s}) = k \right\}. \]
In \cite{han-josuat} Han and Josuat-Verg\`es approached the case in which $c_1=\cdots=c_n=r$\footnote{A caveat is that there is a slight difference in the indexing that the authors of \cite{han-josuat} use for the Eulerian numbers. They write $A_{n,k}$ for what we would denote $A(n,k-1)$.}. They provided formulas and interpretations for $A_{n,k}^{(r,\ldots,r)}$. In particular, to obtain an explicit formula for such numbers, they studied certain half-open polytopes having these numbers as their volumes. We can extend this technique to arbitrary vectors $\mathbf{c}=(c_1,\ldots,c_n)$.

First, for each $v=(v_1,\ldots,v_n)\in \mathbb{R}^n$, we call $\operatorname{Des}(v)=\{i\in [n-1]:v_i>v_{i+1}\}$. Let us use the name 
    \[\operatorname{fdes}_{\mathbf{c}}(v):=v_n+\sum_{i\in \operatorname{Des}(v)} c_{i+1},\]
and define the following two (half-open) polytopes
    \begin{align*}
    \mathcal{F}_{n,k}^{(\mathbf{c})} &:= \left\{v\in\mathscr{R}_{\mathbf{c}} : k \leq \operatorname{fdes}_{\mathbf{c}}(v) < k+1\right\},\\
     \mathcal{A}_{n,k}^{(\mathbf{c})} &:= \left\{ v\in \mathscr{R}_{\mathbf{c}} : k  \leq \sum_{i=1}^n v_i < k+1 \right\}.
    \end{align*}
The key observation is that the above two (half-open) polytopes have the same volume, as there is a measure-preserving map between them that we shall describe. In \cite{stanley-eulerian} Stanley studied the case $c_1=\cdots=c_n=1$ and in \cite{han-josuat} Han and Josuat-Verg\`es extended this to the case $c_1=\cdots=c_n=r$. The latter case had also been approached before by Steingr\'imsson in \cite[Section 4.4]{steingrimsson}. We outline here how to show it for an arbitrary vector $\mathbf{c}$. Consider the map $\varphi:[0,c_1)\times\cdots\times[0,c_n)\to [0,c_1)\times\cdots\times [0,c_n)$ defined everywhere, except in a measure zero set, by $(a_1,\ldots,a_n)\mapsto (b_1,\ldots,b_n)$ where for each $1\leq i\leq n$ we have
    \[ b_i := \left\{ 
    \begin{matrix} 
        a_i - a_{i-1} & \text{ if $a_{i-1}< a_i$}\\
        a_i - a_{i-1} + c_i & \text{ if $a_{i-1}>a_i$}
    \end{matrix}\right.,\]
where we use the convention $a_0:=0$. This map is injective where defined, as if we assume that $(a_1,\ldots,a_n)$ and $(a_1',\ldots,a_n')$ map to the same vector $(b_1,\ldots,b_n)$, then looking at the first coordinates yields $a_1=a_1'$. Next, looking at the second coordinates, we must have
    \[ a_2 - a_1 + \varepsilon c_2 = a_2'-a_1'  + \varepsilon' c_2,\]
where $\varepsilon, \varepsilon'\in \{0,1\}$ depend on how $a_1$ and $a_2$ compare (resp. $a_1'$ and $a_2'$). Cancelling the summand $a_1=a_1'$ from both sides gives $a_2 = a_2' + (\varepsilon'-\varepsilon) c_2$. Since we are assuming that both vectors $(a_1,\ldots,a_n)$ and $(a_1',\ldots,a_n')$ lie in $[0,c_1)\times\cdots\times [0,c_n)$, the only possibility is that $\varepsilon=\varepsilon'$ and that $a_2=a_2'$. Inductively we can continue using the same reasoning.

Notice that the map $\varphi$ is indeed measure preserving, as if we consider for each $\sigma\in\mathfrak{S}_n$ the subset of $\mathscr{R}_{\mathbf{c}}$ defined by 
\[\mathscr{R}_{\mathbf{c},\sigma} := \left\{v\in \mathscr{R}_{\mathbf{c}} : v_{\sigma(1)} < v_{\sigma(2)} < \cdots < v_{\sigma(n)}\right\},\footnote{Notice that this subpolytope is not necessarily a simplex.}\]
then $\varphi|_{\mathscr{R}_{\mathbf{c},\sigma}}$ is just a map of the form $\varphi(\mathbf{x}) = A\mathbf{x} + \mathbf{\varepsilon}$, where $A$ is a matrix with $1$'s in the main diagonal and $-1$'s in second diagonal, whereas $\mathbf{\varepsilon}$ is a vector having $c_i$'s or zeros as entries. 

Moreover, observe that if we take a vector $v\in \mathcal{F}_{n,k}^{(\mathbf{c})}$ having distinct coordinates and we call $(v_1',\ldots,v_n'):=\varphi(v)$, we have that 
    \[v_1'+\cdots+v_n' = (v_1-v_0)+(v_2-v_1)+\cdots+(v_n-v_{n-1}) + \sum_{i\in \operatorname{Des}(v)} c_{i+1} = \operatorname{fdes}_{\mathbf{c}}(v),\]
and hence $\varphi$ maps (up to a measure zero subset) all of $\mathcal{F}_{n,k}^{(\mathbf{c})}$ to  $\mathcal{A}_{n,k}^{(\mathbf{c})}$, and hence the volumes of these two (half-

On one hand we have that $\vol(\mathcal{A}_{n,k}^{(\mathbf{c})})=\vol(\mathcal{F}_{n,k}^{(\mathbf{c})})$ whereas, on the other hand, it is not difficult to show that $\vol(\mathcal{F}^{(\mathbf{c})}_{n,k}) = \frac{1}{n!}A_{n,k}^{(\mathbf{c})}$. In other words, \[\vol(\mathcal{A}_{n,k}^{(\mathbf{c})}) = \frac{1}{n!}\,A_{n,k}^{(\mathbf{c})}.\]

We can use that $\mathcal{A}_{n,k}^{(\mathbf{c})}$ is a slice of a prism with some missing facets. In particular, its volume will not change if we add these facets. This trick leads to a proof of the following result.

\begin{cor}
    The flag Eulerian number $A_{n,k}^{(\mathbf{c})}$ is given by
    \[A_{n,k}^{(\mathbf{c})} = \sum_{\ell = 0}^{k} B(\ell, \mathbf{c}) A(n, k-\ell)\]
    where $B(\ell,\mathbf{c})$ is defined as in Corollary \ref{coromain}.
\end{cor}
\begin{proof}
    Consider the polytope $\mathscr{R}_{k+1, \mathbf{c}'}$ for $\mathbf{c}' = (c_1,\ldots,c_n, 1)\in \mathbb{Z}_{>0}^{n+1}$. By Corollary \ref{coromain}, we obtain
        \begin{equation} A_{n,k}^{(\mathbf{c})} = n! \vol\left(\mathcal{A}^{(\mathbf{c})}_{n,k}\right) = n! \vol\left(\mathscr{R}_{k+1,\mathbf{c}'}\right) =\sum_{\ell = 0}^{k} B(\ell, \mathbf{c}') A(n, k-\ell),\end{equation}
    where $B(\ell, \mathbf{c}')$ is in this case the number of ways of placing $\ell$ balls into $n+1$ boxes of capacities $c_1-1,\ldots,c_n-1$ and $0$ or, in other words, $B(\ell,\mathbf{c}')=B(\ell,\mathbf{c})$, which completes the proof. 
\end{proof}

The identity in the preceding Corollary provides a new combinatorial formula for the version of the flag Eulerian numbers studied in \cite{han-josuat}, by just particularizing $c_1=\cdots=c_n=r$. We do not know of a direct combinatorial way of deriving the preceding equality.

Additionally, since the (normalized) volume of a polytope can be obtained by evaluating the $h^*$-polynomial in $x=1$, we have a refinement of $A_{n,k}^{(\mathbf{c})}$ by considering the coefficients of the $h^*$-polynomial of the polytope $\mathscr{R}_{k+1,\mathbf{c}'}$ by considering the winding numbers.

\begin{cor}
    The flag Eulerian number $A_{n,k}^{(\mathbf{c})}$ is given by
    \[A_{n,k}^{(\mathbf{c})} = \#\{\text{$\mathbf{c}'$-compatible decorated ordered set partitions of type $(k+1,n+1)$}\},\]
    where $\mathbf{c}'=(\mathbf{c},1)\in\mathbb{Z}^{n+1}_{>0}$.
\end{cor}

Compare this result to \cite[Remark 1.2]{kim}. A version of this fact for hypersimplicial permutations (i.e. $\mathbf{c}=(1,\ldots,1)$) appeared also in \cite{ocneanu}.

\subsection{Independence polytopes of uniform matroids}

The uniform matroid $\mathsf{U}_{k,n}$ of rank $k$ and cardinality $n$ has the hypersimplex $\Delta_{k,n}$ as its base polytope. Its independence polytope, namely, the convex hull of the indicator vectors of all the independent subsets, denoted $\mathscr{P}_I(\mathsf{U}_{k,n})$ is given by
    \begin{equation}
        \mathscr{P}_I(\mathsf{U}_{k,n}) = \left\{ x\in [0,1]^n : \sum_{i=1}^n x_i \leq k\right\}.
    \end{equation}
In his PhD thesis \cite{duna} Duna conjectured that these polytopes are Ehrhart positive, and he proved that they indeed are when $k=2$. Later, Ferroni proved in \cite{ferroni2} that Ehrhart positivity holds for all $k$ and $n$, but relying on complicated inequalities. 

Observe that this polytope is a fat slice of the unit cube. Therefore, using Proposition~\ref{prop:fat-thin}, we see that 

    \[ \ehr(\mathscr{P}_I(\mathsf{U}_{k,n}), t) = \ehr(\mathscr{R}_{k,\mathbf{c}'},t),\]
where $\mathbf{c}' = (1,\ldots, 1, k)$. This provides an alternative and much more conceptual proof of the Ehrhart positivity of these polytopes. Moreover, we can retrieve an explicit formula for their Ehrhart polynomial, also described in \cite[Theorem~4.9]{duna}.

\begin{cor}
    The Ehrhart polynomial of $\mathscr{P}_I(\mathsf{U}_{k,n})$ is given by
    \[ \ehr(\mathscr{P}_I(\mathsf{U}_{k,n}), t) = \sum_{j=0}^{k-1}(-1)^j \binom{n}{j} \binom{(k-j)t+n-j}{n}.\]
    This polynomial has positive coefficients.
\end{cor}

\begin{proof}
    Since our polytope is integrally equivalent to $\mathscr{R}_{k,\mathbf{c}'}$ for $\mathbf{c}'=(1,\ldots,1,k)\in \mathbb{Z}_{>0}^{n+1}$, using Theorem~\ref{thm:first-formula}, we have
    \[ \ehr(\mathscr{P}_I(\mathsf{U}_{k,n}), t) = \sum_{j=0}^{k-1}(-1)^j \sum_{v=0}^{k-1} \binom{(k-v)t+(n+1)-1-j}{(n+1)-1}\rho_{\mathbf{c},j}(v).\]
    Recalling from equation \eqref{eq:def-rho} the definition of $\rho_{\mathbf{c},j}(v)$, we have that it is the number of ways of summing $v$ using exactly $j$ numbers among the $c_i$'s. However, the only number $v$ that can be obtained as the sum of $j$ values among the $c_i$'s is just $v=j$, in exactly $\binom{n}{j}$ ways. In other words, $\rho_{\mathbf{c},j}(v) = \binom{n}{j}$ if $j=v$ and it is zero otherwise. This yields the result.
\end{proof}

\subsection{Hilbert series of algebras of Veronese type}

The following algebraic structures arise naturally in combinatorics and commutative algebra.

\begin{defi}
    Let $\mathbb{F}$ be a field, and fix a vector $\mathbf{c}=(c_1,\ldots,c_n)\in\mathbb{Z}^n_{>0}$ and a positive integer $k<c_1+\cdots+c_n$. The graded subalgebra of $\mathbb{F}[x_1,\ldots,x_n]$ generated by all the monomials $x_1^{\alpha_1}\cdots x_n^{\alpha_n}$ where $\alpha_1+\cdots+\alpha_n=k$ and $\alpha_i\leq c_i$ for each $1\leq i\leq n$ is said to be of \emph{Veronese type} and is denoted by $\mathscr{V}(\mathbf{c},k)$.
\end{defi}

In \cite{denegri-hibi} De Negri and Hibi studied the algebras of Veronese type and characterized the scenarios in which they are Gorenstein. In particular, in \cite[Corollary 2.2]{denegri-hibi} they prove that these algebras are Cohen-Macaulay and normal. An important step in such proofs consists of showing that there is an isomorphism between $\mathscr{V}(\mathbf{c},k)$ and the Ehrhart ring of $\mathscr{R}_{k,\mathbf{c}}$ (see \cite{bruns-herzog} for detailed definitions on Ehrhart rings of polytopes). In particular, we obtain as a consequence of Theorem~\ref{main} the following result.

\begin{cor}
    The Hilbert function of an algebra of Veronese type $\mathscr{A}$ is a polynomial having positive coefficients.
\end{cor}

\begin{proof}
    By definition, the Hilbert function of $\mathscr{A}=\mathscr{V}(\mathbf{c},k)$ is the map:
        \[ m \mapsto \dim \mathscr{A}^m. \]
    By the discussion above, the dimension of the graded component $\mathscr{A}^m$ equals $\ehr(\mathscr{R}_{k,\mathbf{c}},m)$, which we know has positive coefficients.
\end{proof}

Notice that Hilbert functions very rarely possess such property. In \cite[p.~1145]{katzman} Katzman referred to the combinatorial descriptions of the $h$-vector of $\mathscr{V}(\mathbf{c},k)$ as ``forbidding''. As a consequence of Theorem~\ref{main-hstar} we now state an interpretation for each of these entries.

\begin{cor}
    The $i$-th entry of the $h$-vector of $\mathscr{V}(\mathbf{c},k)$ counts the number of $\mathbf{c}$-compatible decorated ordered permutations of type $(k,n)$ and winding number $i$.
\end{cor}

\section*{Acknowledgements}

The first author wants to thank Katharina Jochemko for valuable suggestions that improved some aspects of this article. He also wants to thank the organizers and participants of the Workshops ``Characteristic Polynomials of Hyperplane Arrangements and Ehrhart Polynomials of Convex Polytopes'' at RIMS Kyoto, and ``Combinatorial and Algebraic Aspects of Lattice Polytopes'' at the Kwansei Gakuin University, both held in February 2023 for several enlightening discussions.

\bibliography{bibliography}

\newcommand{\etalchar}[1]{$^{#1}$}
\begin{thebibliography}{HMM{\etalchar{+}}22}

\bibitem[BH93]{bruns-herzog}
Winfried Bruns and J\"{u}rgen Herzog.
\newblock {\em Cohen-{M}acaulay rings}, volume~39 of {\em Cambridge Studies in
  Advanced Mathematics}.
\newblock Cambridge University Press, Cambridge, 1993.

\bibitem[BL21]{braun-liu}
Benjamin Braun and Fu~Liu.
\newblock {$h$}*-polynomials with roots on the unit circle.
\newblock {\em Exp. Math.}, 30(3):332--348, 2021.

\bibitem[BR15]{beck-robins}
Matthias Beck and Sinai Robins.
\newblock {\em Computing the continuous discretely}.
\newblock Undergraduate Texts in Mathematics. Springer, New York, second
  edition, 2015.
\newblock Integer-point enumeration in polyhedra, With illustrations by David
  Austin.

\bibitem[Bra16]{braun}
Benjamin Braun.
\newblock Unimodality problems in {E}hrhart theory.
\newblock In {\em Recent trends in combinatorics}, volume 159 of {\em IMA Vol.
  Math. Appl.}, pages 687--711. Springer, [Cham], 2016.

\bibitem[BVV97]{bruns}
Winfried Bruns, Wolmer~V. Vasconcelos, and Rafael~H. Villarreal.
\newblock Degree bounds in monomial subrings.
\newblock {\em Illinois J. Math.}, 41(3):341--353, 1997.

\bibitem[CL18]{castillo-liu}
Federico Castillo and Fu~Liu.
\newblock Berline-{V}ergne valuation and generalized permutohedra.
\newblock {\em Discrete Comput. Geom.}, 60(4):885--908, 2018.

\bibitem[CL21]{castillo-liu2}
Federico Castillo and Fu~Liu.
\newblock On the {T}odd class of the permutohedral variety.
\newblock {\em Algebr. Comb.}, 4(3):387--407, 2021.

\bibitem[DLHK09]{deloera-haws-koppe}
Jes\'{u}s~A. De~Loera, David~C. Haws, and Matthias K\"{o}ppe.
\newblock Ehrhart polynomials of matroid polytopes and polymatroids.
\newblock {\em Discrete Comput. Geom.}, 42(4):670--702, 2009.

\bibitem[DNH97]{denegri-hibi}
Emanuela De~Negri and Takayuki Hibi.
\newblock Gorenstein algebras of {V}eronese type.
\newblock {\em J. Algebra}, 193(2):629--639, 1997.

\bibitem[Dun19]{duna}
Ken Duna.
\newblock {\em Matroid {I}ndependence {P}olytopes and {T}heir {E}hrhart
  {T}heory}.
\newblock ProQuest LLC, Ann Arbor, MI, 2019.
\newblock Thesis (Ph.D.)--University of Kansas.

\bibitem[{Ear}17]{early}
Nick {Early}.
\newblock {Conjectures for Ehrhart $h^*$-vectors of Hypersimplices and Dilated
  Simplices}.
\newblock {\em arXiv e-prints}, page arXiv:1710.09507, October 2017.

\bibitem[Ehr62]{ehrhart}
Eug\`ene Ehrhart.
\newblock Sur les poly\`edres rationnels homoth\'{e}tiques \`a {$n$}
  dimensions.
\newblock {\em C. R. Acad. Sci. Paris}, 254:616--618, 1962.

\bibitem[Fer21a]{ferroni1}
Luis Ferroni.
\newblock Hypersimplices are {E}hrhart positive.
\newblock {\em J. Combin. Theory Ser. A}, 178:Paper No. 105365, 13, 2021.

\bibitem[{F}er21b]{ferroni2}
Luis {F}erroni.
\newblock Integer point enumeration on independence polytopes and half-open
  hypersimplices.
\newblock {\em Discrete Math.}, 344(8):Paper No. 112446, 6, 2021.

\bibitem[{Fer}22]{ferroni3}
Luis {Ferroni}.
\newblock Matroids are not {E}hrhart positive.
\newblock {\em Adv. Math.}, 402:Paper No. 108337, 27, 2022.

\bibitem[FJS22]{fjs}
Luis Ferroni, Katharina Jochemko, and Benjamin Schr\"{o}ter.
\newblock Ehrhart polynomials of rank two matroids.
\newblock {\em Adv. in Appl. Math.}, 141:Paper No. 102410, 26, 2022.

\bibitem[GS20]{solus-gustafsson}
Nils Gustafsson and Liam Solus.
\newblock Derangements, {E}hrhart theory, and local {$h$}-polynomials.
\newblock {\em Adv. Math.}, 369:107169, 35, 2020.

\bibitem[HH02]{herzog-hibi}
J\"{u}rgen Herzog and Takayuki Hibi.
\newblock Discrete polymatroids.
\newblock {\em J. Algebraic Combin.}, 16(3):239--268 (2003), 2002.

\bibitem[HHV05]{herzog-hibi-vladoiu}
J\"{u}rgen Herzog, Takayuki Hibi, and Marius Vladoiu.
\newblock Ideals of fiber type and polymatroids.
\newblock {\em Osaka J. Math.}, 42(4):807--829, 2005.

\bibitem[HJV16]{han-josuat}
Guo-Niu Han and Matthieu Josuat-Verg\`es.
\newblock Flag statistics from the {E}hrhart {$h^*$}-polynomial of
  multi-hypersimplices.
\newblock {\em Electron. J. Combin.}, 23(1):Paper 1.55, 20, 2016.

\bibitem[HMM{\etalchar{+}}22]{hanelyetal}
Derek {Hanely}, Jeremy~L. {Martin}, Daniel {McGinnis}, Dane {Miyata}, George~D.
  {Nasr}, Andr{\'e}s~R. {Vindas-Mel{\'e}ndez}, and Mei {Yin}.
\newblock {Ehrhart Theory of Paving and Panhandle Matroids}.
\newblock {\em arXiv e-prints}, page arXiv:2201.12442, January 2022.

\bibitem[HZ19a]{haglund-zhang}
James Haglund and Philip~B. Zhang.
\newblock Real-rootedness of variations of {E}ulerian polynomials.
\newblock {\em Adv. in Appl. Math.}, 109:38--54, 2019.

\bibitem[HZ19b]{herzog-zhu}
J\"{u}rgen Herzog and Guangjun Zhu.
\newblock Freiman ideals.
\newblock {\em Comm. Algebra}, 47(1):407--423, 2019.

\bibitem[JR22]{jochemko-ravichandran}
Katharina Jochemko and Mohan Ravichandran.
\newblock Generalized permutahedra: {M}inkowski linear functionals and
  {E}hrhart positivity.
\newblock {\em Mathematika}, 68(1):217--236, 2022.

\bibitem[Kat05]{katzman}
Mordechai Katzman.
\newblock The {H}ilbert series of algebras of the {V}eronese type.
\newblock {\em Comm. Algebra}, 33(4):1141--1146, 2005.

\bibitem[Kim20]{kim}
Donghyun Kim.
\newblock A combinatorial formula for the {E}hrhart {$h^\ast$}-vector of the
  hypersimplex.
\newblock {\em J. Combin. Theory Ser. A}, 173:105213, 15, 2020.

\bibitem[Li12]{li}
Nan Li.
\newblock Ehrhart {$h^*$}-vectors of hypersimplices.
\newblock {\em Discrete Comput. Geom.}, 48(4):847--878, 2012.

\bibitem[Liu19]{liu}
Fu~Liu.
\newblock On positivity of {E}hrhart polynomials.
\newblock In {\em Recent trends in algebraic combinatorics}, volume~16 of {\em
  Assoc. Women Math. Ser.}, pages 189--237. Springer, Cham, 2019.

\bibitem[LP07]{lam-postnikov}
Thomas Lam and Alexander Postnikov.
\newblock Alcoved polytopes. {I}.
\newblock {\em Discrete Comput. Geom.}, 38(3):453--478, 2007.

\bibitem[LP20]{lam-postnikov-polypositroids}
Thomas Lam and Alexander Postnikov.
\newblock {Polypositroids}.
\newblock {\em arXiv e-prints}, page arXiv:2010.07120, October 2020.

\bibitem[LT19]{liu-tsuchiya}
Fu~Liu and Akiyoshi Tsuchiya.
\newblock Stanley's non-{E}hrhart-positive order polytopes.
\newblock {\em Adv. in Appl. Math.}, 108:1--10, 2019.

\bibitem[McG23]{mcginnis2023combinatorial}
Daniel McGinnis.
\newblock {A combinatorial formula for the Ehrhart coefficients of a certain
  class of weighted multi-hypersimplices}.
\newblock {\em arXiv e-prints}, page arXiv:2303.04113, March 2023.

\bibitem[McM77]{mcmullen}
P.~McMullen.
\newblock Valuations and {E}uler-type relations on certain classes of convex
  polytopes.
\newblock {\em Proc. London Math. Soc. (3)}, 35(1):113--135, 1977.

\bibitem[{Ocn}13]{ocneanu}
Adrian {Ocneanu}.
\newblock {On the inner structure of a permutation: Bicolored Partitions and
  Eulerians, Trees and Primitives}.
\newblock {\em arXiv e-prints}, page arXiv:1304.1263, April 2013.

\bibitem[Pos09]{postnikov}
Alexander Postnikov.
\newblock Permutohedra, associahedra, and beyond.
\newblock {\em Int. Math. Res. Not. IMRN}, (6):1026--1106, 2009.

\bibitem[SS21]{sinn-sjoberg}
Rainer {Sinn} and Hannah {Sj{\"o}berg}.
\newblock {Do alcoved lattice polytopes have unimodal h*-vector?}
\newblock {\em arXiv e-prints}, page arXiv:2104.15080, April 2021.

\bibitem[Sta77]{stanley-eulerian}
Richard~P. Stanley.
\newblock {Eulerian partitions of a unit hypercube}.
\newblock {Higher Comb., Proc. NATO Adv. Study Inst., Berlin (West) 1976, 49
  (1977).}, 1977.

\bibitem[Sta93]{stanley-hstar}
Richard~P. Stanley.
\newblock A monotonicity property of {$h$}-vectors and {$h^*$}-vectors.
\newblock {\em European J. Combin.}, 14(3):251--258, 1993.

\bibitem[Sta12]{stanley-combinatorics1}
Richard~P. Stanley.
\newblock {\em Enumerative combinatorics. {V}olume 1}, volume~49 of {\em
  Cambridge Studies in Advanced Mathematics}.
\newblock Cambridge University Press, Cambridge, second edition, 2012.

\bibitem[Ste94]{steingrimsson}
Einar Steingr\'{\i}msson.
\newblock Permutation statistics of indexed permutations.
\newblock {\em European J. Combin.}, 15(2):187--205, 1994.

\bibitem[Stu96]{sturmfels}
Bernd Sturmfels.
\newblock {\em Gr\"{o}bner bases and convex polytopes}, volume~8 of {\em
  University Lecture Series}.
\newblock American Mathematical Society, Providence, RI, 1996.

\end{thebibliography}
\bibliographystyle{alpha}

\end{document}